\documentclass[11pt,twoside]{article}

\newcommand{\HeadTitle}{Compact Coefficient Formulae for Logarithmic Tangent and Hyperbolic Integrals}

\newcommand{\HeadTitleTwo}{\begin{center}
\Large{\textit{Compact Coefficient Formulae for Logarithmic Tangent and Hyperbolic Integrals}}
\end{center}}

\usepackage[margin=0.8in]{geometry}
\usepackage{amsmath, amssymb, amsfonts, mathtools}
\usepackage{amsthm}
\usepackage{mathrsfs}
\usepackage{stmaryrd}
\usepackage{esint}
\usepackage{eqnalign}

\usepackage{graphicx}
\usepackage{enumitem}
\usepackage[most]{tcolorbox}
\usepackage{float}
\usepackage{listings}
\usepackage{array}
\usepackage{booktabs}
\usepackage{xcolor}
\usepackage{emptypage}
\usepackage{tensor}
\usepackage{tabularx}

\usepackage{fancyhdr}
\usepackage{hyperref}
\usepackage[nameinlink,noabbrev]{cleveref}

\usepackage{titlesec}
\usepackage[T1]{fontenc}
\usepackage[utf8]{inputenc}
\usepackage{lmodern}
\usepackage{titling}

\usepackage[backend=biber,style=numeric]{biblatex}
\providecommand{\HeadTitle}{} 
\providecommand{\HeadTitleTwo}{} 
\providecommand{\HeadAuthor}{Luc Ramsès TALLA WAFFO} 

\titleformat{\section}[block]
  {\normalfont\large\bfseries\itshape\centering}
  {§\thesection.}
  {1em}
  {}

\titleformat{\subsection}
  {\normalfont\normalsize\bfseries}
  {\thesubsection}
  {1em}
  {}

\newtheorem{theorem}{Theorem}[section]
\newtheorem{lemma}[theorem]{Lemma}
\newtheorem{proposition}[theorem]{Proposition}
\newtheorem{corollary}[theorem]{Corollary}

\newtheorem{remark}[theorem]{Remark}

\newtheorem{example}[theorem]{Example}

\crefname{example}{example}{examples}
\Crefname{example}{Example}{Examples}

\crefname{corollary}{corollary}{corollaries}
\Crefname{corollary}{Corollary}{Corollaries}

\crefname{definition}{definition}{definitions}
\Crefname{definition}{Definition}{Definitions}

\crefname{remark}{remark}{remarks}
\Crefname{remark}{Remark}{Remarks}

\crefname{conjecture}{conjecture}{conjectures}
\Crefname{conjecture}{Conjecture}{Conjectures}

\crefname{lemma}{lemma}{lemmas}
\Crefname{lemma}{Lemma}{Lemmas}

\crefname{proposition}{proposition}{propositions}
\Crefname{proposition}{Proposition}{Propositions}

\crefname{theorem}{theorem}{theorems}
\Crefname{theorem}{Theorem}{Theorems}

\numberwithin{equation}{section}

\usepackage{fontspec}
\begin{document}

\thispagestyle{fancy}

\vspace{0.2cm}

\begin{center}
\Large{\HeadTitleTwo}
\end{center}

\hspace{3cm}

\begin{center}
Luc Ramsès TALLA WAFFO \\
Technische Universität Darmstadt\\
Karolinenplatz 5, 64289 Darmstadt, Germany\\
ramses.talla@stud.tu-darmstadt.de\\
\vspace{0.5cm}
\today
\end{center}

\begin{abstract}
We develop compact coefficient-extraction formulae for several families of
hyperbolic, logarithmic tangent, and Malmsten-type integrals whose values are
finite linear combinations of odd zeta values and even Dirichlet beta values.
The principal advantage of these formulae is that coefficients previously
encoded by recursive arrays or nested finite sums are replaced by a single
coefficient of an explicit elementary expression.  This makes the
coefficients easier to compute, keeps the dependence on the parameters
visible, and reveals structural features---such as vanishing ranges,
extremal coefficients, and sign patterns---without hidden cancellations.
For shifted hyperbolic integrals with numerator $\sinh((2k+1)x)$, the
coefficients are expressed through Chebyshev--arcsine extractions.  The same
mechanism yields Laurent coefficient formulae for logarithmic tangent
integrals and leads to direct proofs of simple initial and terminal
coefficients, including a parity-free terminal identity.  For $m,n\geq1$,
$m\geq n$, and $m+n$ even, we prove
\[
\int_0^\infty\frac{\tanh^{m+1}x}{x^{n+1}}\,dx
=
(-1)^{(m-n)/2}
\sum_{p=\lceil n/2\rceil}^{(m+n)/2}
\binom{2p}{n}
(2^{2p+1}-1)
\frac{\zeta(2p+1)}{\pi^{2p}}
[u^{m+n-2p}](u\cot u)^{m+1}.
\]
\[\int_0^\infty
\frac{\tanh^m x}{x^n\cosh x}\,dx
=
(-1)^{(m-n)/2}
\sum_{p=\lceil n/2\rceil}^{(m+n)/2}
2^{2p}\binom{2p-1}{n-1}
\frac{\beta(2p)}{\pi^{2p-1}}
[u^{m+n-2p}]
\frac{u}{\sin u}(u\cot u)^m
\]
In the diagonal zeta case, the formula gives
the family $\int_0^\infty(\tanh x/x)^N\,dx$ in a direct, non-recursive form
and makes the disappearance of the initial zeta values immediate.
\end{abstract}

\vspace{0.2cm}

\paragraph{Notation.}
We write \(\mathbb N=\{1,2,\ldots\}\), \(\mathbb N_0=\{0,1,2,\ldots\}\).  The Riemann zeta function and the
Dirichlet beta function are denoted by \(\zeta(s)\) and \(\beta(s)\),
respectively.  The Chebyshev polynomial of the second kind \cite{MasonHandscomb2003}
\cite{Mason1993}\cite{Rivlin1990}
\cite{NIST:DLMF} is denoted by
\(U_m\) and is normalized by\(\displaystyle
U_m(\cos\theta)=\frac{\sin((m+1)\theta)}{\sin\theta}.
\) For a power series or Laurent series \(F\), the notation \([x^r]F(x)\)
denotes the coefficient of \(x^r\) in the expansion at the origin.

\vspace{0.5cm}

\section*{Introduction}

Definite integrals involving hyperbolic functions and logarithmic tangent
kernels often admit finite evaluations in terms of odd Riemann zeta values
and even Dirichlet beta values. In such formulae, the main difficulty frequently
lies not in the special values themselves but in their coefficients, which are
often expressed through recurrences, nested finite sums, or auxiliary arrays.
Although such descriptions can be effective for computation at fixed
parameters, they tend to obscure how the coefficients vary with the
parameters, which special values occur, and which terms vanish. In
particular, cancellations built into a recurrence or a multiple sum may
remain invisible until the full expression has been evaluated. For instance, Kyrion \cite{Kyrion2025}
studied integrals of the form
\(\int_0^\infty (\tanh x/x)\operatorname{sech}^L x\,e^{-Tx}\,dx\) and, for \(T=0\),
obtained recursive coefficient formulae for the zeta and beta expansions
for the two families \(L=2N\) and \(L=2N+1\), respectively.
These results provide a natural point of comparison for the compact
coefficients developed here. Related contour evaluations
for shifted hyperbolic and logarithmic tangent integrals were obtained in
\cite{talla_waffo_integral_2025}.

\vspace{0.1cm}

Particularly close antecedents are due to Jing Li and Wenchang Chu.
They evaluated by contour integration the family with integrand
\(\tanh^M x/x^N\) on \((0,\infty)\), subject to \(M\ge N\ge2\)
and \(M\equiv N\pmod2\), as finite linear combinations of odd
zeta values \cite{LiChu2024Tanh}. Their formula is explicit, but its coefficients
involve several nested finite sums arising from higher-order residues.
With \(M=m+1\) and \(N=n+1\), this is precisely the zeta-type
tangent-power family considered below. The distinction between an integral evaluation and a coefficient formula
is important here: our contribution is not a new evaluation of that
family, but a compact reformulation of its coefficients:
the nested sums are replaced by single extractions of the form
\([u^r](u\cot u)^{m+1}\). In this way, each odd-zeta coefficient is recovered from a classical
generating function. Besides simplifying the expression, the extraction
shows its dependence on \(m,n\), the parity conditions, the admissible
zeta values, and the vanishing ranges directly. Li and Chu also evaluated
three families of trigonometric integrals with \(\ln(\tan x)\) in the
denominator \cite{LiChu2024LogTan}; their full-interval families
\(\Phi(n)\) and \(\Psi(n-1)\) correspond to the two half-interval
families treated here. The interval reduction is therefore not claimed
as a new integral evaluation; the contribution is instead the compression
of the associated zeta and beta coefficients into Chebyshev--arcsine
Laurent extractions.

\vspace{0.1cm}

The guiding principle of the paper is to identify an entire coefficient
family with the coefficients of a single explicit power series or Laurent
series. This replaces a list of recursively or combinatorially defined
constants by one elementary expression, without losing the dependence
on the original parameters. The
resulting formulae use Chebyshev polynomials of the second kind, powers of
\(\arcsin x\), and elementary generating functions such as
\((u\cot u)^r\) and \((w/\sinh w)^r\). They make support, vanishing
ranges, and extremal coefficients accessible without having to resolve cancellations hidden within nested sums. We use standard Chebyshev identities; see, for example,
\cite[Section~4.6]{MasonHandscomb2003}, \cite{Mason1993,Rivlin1990},
and \cite[Chapter~18]{NIST:DLMF}.

\vspace{0.1cm}

We first consider shifted hyperbolic integrals on \((0,\infty)\).
Their numerator is \(\sinh((2k+1)x)\), with denominator
\(x\cosh^{2n+1}x\) or \(x\cosh^{2n}x\).
Their respective zeta and beta coefficients reduce to
\([x^{2n}]U_{2k}(x)(\arcsin x)^{2p}\) and
\([x^{2n-1}]U_{2k}(x)(\arcsin x)^{2p-1}\). For \(k=0\), the formulas reduce to the unshifted families studied by
Kyrion. For general admissible \(k\), the frequency shift is carried
entirely by the Chebyshev factor \(U_{2k}\), so the coefficient
extraction does not require a separate recursive construction for
each shift. Linearizing
odd powers of \(\sinh x\) collapses the corresponding polynomial
combination to \((x^2-1)^k\), yielding formulae for numerators
\(\sinh^{2k+1}x\) and, after \(x=\tanh t\), corollaries involving
\(1/\operatorname{arctanh}x\). Even numerator frequencies
\(\sinh(2kx)\) are handled by the adjacent-frequency identity
\(2\cosh x\,\sinh(2kx)=\sinh((2k+1)x)+\sinh((2k-1)x)\) and the
recurrence \(U_{2k}(z)+U_{2k-2}(z)=2zU_{2k-1}(z)\). In particular, the even-frequency cases are controlled by
\(U_{2k-1}\), just as the odd-frequency cases are controlled by
\(U_{2k}\). Thus both frequency parities fit the same extraction
framework, while the parity of the denominator determines whether odd
zeta or even beta values appear.

\vspace{0.1cm}

For the logarithmic tangent integrals with kernels
\(\sin(4nx)/\ln(\tan x)\) and
\(\cos((4n-2)x)/\ln(\tan x)\) on \((0,\pi/4)\),
multiple-angle and Chebyshev identities yield compact Laurent extractions involving
\(U_{2n-1}(1/x)\) or \(U_{2n-2}(1/x)\) multiplied by powers of
\(\arcsin x\). Besides collecting all coefficients into a single Laurent product, this
approach makes the boundary cases particularly transparent. It gives a
direct coefficient-extraction proof for the first beta coefficient and
combines the terminal coefficients of both trigonometric families into
the parity-independent identity
\([x]U_{r-1}(1/x)(\arcsin x)^r=2^{r-1}\), for every integer \(r\ge1\). Unlike the separate endpoint
formulae, the identity has the same form in the sine and cosine cases.

\vspace{0.1cm}

The same approach extends to tangent-power integrals on
\((0,\infty)\) with kernels \(\tanh^{m+1}x/x^{n+1}\) and
\(\tanh^m x/(x^n\cosh x)\), for \(m,n\ge1\), \(m\ge n\),
and \(m+n\) even.
The two integrals can be treated with the same integration-by-parts
and derivative mechanism. We introduce a single auxiliary
derivative-polynomial family, defined by
\(\frac{d^r}{dx^r}(\tanh^l x/\cosh^q x)
=\cosh^{-q}x\,P_r^{l,q}(\tanh x)\)
rather than separate polynomial arrays for the zeta and beta cases.
The former uses the specialization
\((l,r,q)=(m+1,n,0)\), while the latter uses
\((l,r,q)=(m,n-1,1)\). A common coefficient collapse eliminates
the intermediate finite sums and leads to the kernel
\((u/\sin u)^q(u\cot u)^l\) and to the two formulae
\[
\int_0^\infty\frac{\tanh^{m+1}x}{x^{n+1}}\,dx
=
(-1)^{(m-n)/2}
\sum_{p=\lceil n/2\rceil}^{(m+n)/2}
\binom{2p}{n}
(2^{2p+1}-1)
\frac{\zeta(2p+1)}{\pi^{2p}}
[u^{m+n-2p}](u\cot u)^{m+1}.
\]
\[\int_0^\infty
\frac{\tanh^m x}{x^n\cosh x}\,dx
=
(-1)^{(m-n)/2}
\sum_{p=\lceil n/2\rceil}^{(m+n)/2}
2^{2p}\binom{2p-1}{n-1}
\frac{\beta(2p)}{\pi^{2p-1}}
[u^{m+n-2p}]
\frac{u}{\sin u}(u\cot u)^m
\]
Both identities retain the original \(m,n\)-parameterization; the
parameters \(l,r,q\) serve only to unify the proof. In particular,
the diagonal case \(m=n\) gives a direct, non-recursive expression for
\(\int_0^\infty (\tanh x/x)^{n+1}\,dx\), with the disappearance
of the initial zeta values explicit. The specialization
\((m,n)=(2N-1,1)\) also recovers the evaluation of
\(\int_0^\infty \tanh^{2N}x/x^2\,dx\) in
\cite{TallaWaffo2025arxiv2511.02843}.

\vspace{0.1cm}

We also apply the method to Blagouchine's logarithmic hyperbolic
integrals \(\mathcal I_r(a,b)=\int_0^\infty
\ln(x^2+a^2)/\cosh^r(bx)\,dx\). In \cite{Blagouchine2014}, the
even- and odd-power cases were described using rational coefficients
tabulated for fixed indices. The coefficients depend on the parity
of \(r\), but both sets come from the same elementary series.
We make them explicit for every admissible index by extracting
coefficients from
\((w/\sinh w)^r\).

\vspace{0.1cm}

Finally, for Malmsten-type integrals representing \(\beta(2n)\)
and \(\zeta(2n+1)\), we prove rigorously the compact
polylogarithmic representations obtained in
\cite{TallaWaffo2025arxiv2511.02843}, using finite Eulerian
expansions rather than an interchange of an infinite series with the
integral. The proof proceeds through finite Chebyshev sums weighted by
Eulerian numbers of types \(A\) and \(B\), and the crucial step
is an Eulerian--Chebyshev coefficient collapse. This collapse
shows that all off-diagonal terms \(p\ne n\) vanish: for \(p>n\)
this follows from the order at the origin, whereas for \(p<n\)
it results from an exact cancellation among terms that need not
vanish individually. Hence only the diagonal term survives, producing
a single beta or zeta value. After normalization, the two
identities can be interpreted as a triangular biorthogonality relation between
powers of \(\arcsin z\) and Eulerian--Chebyshev polynomial families.

\vspace{0.1cm}

The paper is organized as follows.
\Cref{sec:coefficient-formulae,sec:coefficient-formulae_double_angle}
treat odd- and even-frequency shifted hyperbolic integrals, including
odd powers of \(\sinh x\) and hyperbolic-arctangent corollaries.
\Cref{sec:coefficient-formulae-logarithmic-tangent-integrals}
develops the Laurent formulae and boundary coefficients, while
\cref{section:general-tangent-integral} establishes the unified
derivative-polynomial method for tangent powers.
\Cref{sec:blagouchine-coefficients} gives Blagouchine's explicit
coefficient formulae, and \cref{sec:Malmsten-integrals} proves the
Malmsten-type identities and their Eulerian--Chebyshev collapse.

\vspace{0.3cm}

\section{Coefficient Formulae for $\displaystyle\int_{0}^{\infty}
\frac{\sinh((2k+1)x)}{x\cosh^{n}x}\,dx$ and $\displaystyle\int_{0}^{\infty}
\frac{\sinh^{2k+1}x}{x\cosh^{n}x}\,dx$}\label[section]{sec:coefficient-formulae}

\begin{lemma}[Conversion of the \(c\)- and \(d\)-sums]
\label[lemma]{lemma:cd-sums-to-arcsine-chebyshev}
Let \(n\in\mathbb N\), \(0\le k<n\), and \(1\le p\le n\). Let
\[
\left(\frac{t}{\sinh t}\right)^{2n+1}
=
\sum_{j\ge0}c_{j,n}t^j,
\qquad
\left(\frac{t}{\sinh t}\right)^{2n}
=
\sum_{j\ge0}d_{j,n}t^j.
\]
Let \(U_m\) denote the Chebyshev polynomial of the second kind. Then
\begin{equation}\label{eq:c-sum-to-arcsine-chebyshev}
\sum_{m=p}^{n}
c_{2n-2m,n}
\frac{(2k+1)^{2m-2p}}{(2m-2p)!}
=
(-1)^{n+p+k}
[x^{2n}]\,U_{2k}(x)(\arcsin x)^{2p},
\end{equation}
\begin{equation}\label{eq:d-sum-to-arcsine-chebyshev}
\sum_{m=0}^{n-p}
d_{2m,n}
\frac{(2k+1)^{2n-2m-2p}}
{(2n-2m-2p)!}
=
(-1)^{n+p+k}
[x^{2n-1}]\,U_{2k}(x)(\arcsin x)^{2p-1}.
\end{equation}
\end{lemma}

\begin{proof}
By the Cauchy product with
\[
\cosh((2k+1)t)
=
\sum_{\ell\ge0}
\frac{(2k+1)^{2\ell}t^{2\ell}}{(2\ell)!},
\]
we obtain
\begin{equation}\label{eq:cd-cauchy-product-local}
\begin{cases}
\displaystyle
\sum_{m=p}^{n}
c_{2n-2m,n}
\frac{(2k+1)^{2m-2p}}{(2m-2p)!}
=
[t^{2n-2p}]
\left(\frac{t}{\sinh t}\right)^{2n+1}
\cosh((2k+1)t),
\\[2.2ex]
\displaystyle
\sum_{m=0}^{n-p}
d_{2m,n}
\frac{(2k+1)^{2n-2m-2p}}
{(2n-2m-2p)!}
=
[t^{2n-2p}]
\left(\frac{t}{\sinh t}\right)^{2n}
\cosh((2k+1)t).
\end{cases}
\end{equation}
By Cauchy's coefficient formula and the substitution \(t=iz\),
\eqref{eq:cd-cauchy-product-local} becomes
\begin{equation}\label{eq:cd-contour-local}
\begin{cases}
\displaystyle
\sum_{m=p}^{n}
c_{2n-2m,n}
\frac{(2k+1)^{2m-2p}}{(2m-2p)!}
=
(-1)^{n+p}
\frac{1}{2\pi i}
\oint_{\Gamma}
\left(\frac{z}{\sin z}\right)^{2n+1}
\cos((2k+1)z)
\frac{dz}{z^{2n-2p+1}},
\\[2.2ex]
\displaystyle
\sum_{m=0}^{n-p}
d_{2m,n}
\frac{(2k+1)^{2n-2m-2p}}
{(2n-2m-2p)!}
=
(-1)^{n+p}
\frac{1}{2\pi i}
\oint_{\Gamma}
\left(\frac{z}{\sin z}\right)^{2n}
\cos((2k+1)z)
\frac{dz}{z^{2n-2p+1}},
\end{cases}
\end{equation}
where \(\Gamma\) is a sufficiently small positively oriented circle around
\(0\).

Now put \(
x=\sin z.
\)
Then \(z=\arcsin x\) and \(dz=dx/\sqrt{1-x^2}\). Hence
\eqref{eq:cd-contour-local} becomes
\begin{equation}\label{eq:cd-after-sine-substitution-local}
\begin{cases}
\displaystyle
\sum_{m=p}^{n}
c_{2n-2m,n}
\frac{(2k+1)^{2m-2p}}{(2m-2p)!}
=
(-1)^{n+p}
\frac{1}{2\pi i}
\oint_{\sin\Gamma}
\frac{\cos((2k+1)\arcsin x)}{\sqrt{1-x^2}}
\frac{(\arcsin x)^{2p}}{x^{2n+1}}\,dx,
\\[2.2ex]
\displaystyle
\sum_{m=0}^{n-p}
d_{2m,n}
\frac{(2k+1)^{2n-2m-2p}}
{(2n-2m-2p)!}
=
(-1)^{n+p}
\frac{1}{2\pi i}
\oint_{\sin\Gamma}
\frac{\cos((2k+1)\arcsin x)}{\sqrt{1-x^2}}
\frac{(\arcsin x)^{2p-1}}{x^{2n}}\,dx.
\end{cases}
\end{equation}

The factor
\(
\dfrac{\cos((2k+1)\arcsin x)}{\sqrt{1-x^2}}
\)
is exactly the Chebyshev factor. Indeed, if \(x=\sin\theta\), then
\[
U_{2k}(x)
=
U_{2k}(\sin\theta)
=
\frac{\sin((2k+1)(\pi/2-\theta))}{\cos\theta}
=
(-1)^k
\frac{\cos((2k+1)\theta)}{\cos\theta}.
\]
Thus
\[
\frac{\cos((2k+1)\arcsin x)}{\sqrt{1-x^2}}
=
(-1)^k U_{2k}(x).
\]
Therefore \eqref{eq:cd-after-sine-substitution-local} becomes
\[
\begin{cases}
\displaystyle
\sum_{m=p}^{n}
c_{2n-2m,n}
\frac{(2k+1)^{2m-2p}}{(2m-2p)!}
=
(-1)^{n+p+k}
\frac{1}{2\pi i}
\oint_{\sin\Gamma}
U_{2k}(x)(\arcsin x)^{2p}
\frac{dx}{x^{2n+1}},
\\[2.2ex]
\displaystyle
\sum_{m=0}^{n-p}
d_{2m,n}
\frac{(2k+1)^{2n-2m-2p}}
{(2n-2m-2p)!}
=
(-1)^{n+p+k}
\frac{1}{2\pi i}
\oint_{\sin\Gamma}
U_{2k}(x)(\arcsin x)^{2p-1}
\frac{dx}{x^{2n}}.
\end{cases}
\]
By Cauchy's coefficient formula, this is exactly
\eqref{eq:c-sum-to-arcsine-chebyshev} and
\eqref{eq:d-sum-to-arcsine-chebyshev}.
\end{proof}

\begin{proposition}[Chebyshev--arcsine form of the shifted integrals]
\label[proposition]{prop:shifted_integrals_chebyshev_arcsine}
Let \(n\in\mathbb N\), \(0\le k<n\), and let \(U_m\) denote the Chebyshev
polynomial of the second kind.Then
\begin{equation}\label{eq:shifted-zeta-chebyshev-arcsine}
\int_{0}^{\infty}
\frac{\sinh((2k+1)x)}{x\cosh^{2n+1}x}\,dx
=
\sum_{p=1}^{n}
(2^{2p+1}-1)
[z^{2n}]\,U_{2k}(z)(\arcsin z)^{2p}
\frac{\zeta(2p+1)}{\pi^{2p}},
\end{equation}
\begin{equation}\label{eq:shifted-beta-chebyshev-arcsine}
\int_{0}^{\infty}
\frac{\sinh((2k+1)x)}{x\cosh^{2n}x}\,dx
=
\sum_{p=1}^{n}
2^{2p}
[z^{2n-1}]\,U_{2k}(z)(\arcsin z)^{2p-1}
\frac{\beta(2p)}{\pi^{2p-1}}.
\end{equation}
\end{proposition}

\begin{proof}
We start from the contour evaluations 
\cite{talla_waffo_integral_2025}:
\begin{equation}\label{eq:talla-waffo-shifted-start}
\begin{cases}
\displaystyle
\int_{-\infty}^{\infty}
\frac{\sinh((2k+1)x)}{x\cosh^{2n+1}x}\,dx
=
\sum_{p=1}^{n}
\alpha_{p,k,n}
\frac{\zeta(2p+1)}{\pi^{2p}},
\\[2.2ex]
\displaystyle
\int_{-\infty}^{\infty}
\frac{\sinh((2k+1)x)}{x\cosh^{2n}x}\,dx
=
\sum_{p=1}^{n}
f_{p,k,n}
\frac{\beta(2p)}{\pi^{2p-1}},
\end{cases}
\end{equation}
where
\begin{equation}\label{eq:talla-waffo-coefficients}
\begin{cases}
\displaystyle
\alpha_{p,k,n}
=
(-1)^{p+n+k}4^{p+1}
\left(1-\frac1{2^{2p+1}}\right)
\sum_{m=p}^{n}
c_{2n-2m,n}
\frac{(2k+1)^{2m-2p}}{(2m-2p)!},
\\[2.2ex]
\displaystyle
f_{p,k,n}
=
(-1)^{n+k+p}2^{2p+1}
\sum_{m=0}^{n-p}
d_{2m,n}
\frac{(2k+1)^{2n-2m-2p}}
{(2n-2m-2p)!}.
\end{cases}
\end{equation}

By \cref{lemma:cd-sums-to-arcsine-chebyshev}, we have
\[
\sum_{m=p}^{n}
c_{2n-2m,n}
\frac{(2k+1)^{2m-2p}}{(2m-2p)!}
=
(-1)^{n+p+k}
[x^{2n}]\,U_{2k}(x)(\arcsin x)^{2p},
\]
\[
\sum_{m=0}^{n-p}
d_{2m,n}
\frac{(2k+1)^{2n-2m-2p}}
{(2n-2m-2p)!}
=
(-1)^{n+p+k}
[x^{2n-1}]\,U_{2k}(x)(\arcsin x)^{2p-1}.
\]
Substituting these two identities into \eqref{eq:talla-waffo-coefficients},
all signs cancel. Hence
\[
\alpha_{p,k,n}
=
4^{p+1}
\left(1-\frac1{2^{2p+1}}\right)
[x^{2n}]\,U_{2k}(x)(\arcsin x)^{2p}.
\]
Since \(\displaystyle
4^{p+1}
\left(1-\frac1{2^{2p+1}}\right)
=
2(2^{2p+1}-1),
\) we get
\[
\alpha_{p,k,n}
=
2(2^{2p+1}-1)
[x^{2n}]\,U_{2k}(x)(\arcsin x)^{2p}.
\]
\[
f_{p,k,n}
=
2^{2p+1}
[x^{2n-1}]\,U_{2k}(x)(\arcsin x)^{2p-1}.
\]

The integrands in \eqref{eq:talla-waffo-shifted-start} are even. Therefore
the integrals over \((0,\infty)\) are one half of the corresponding
integrals over \(\mathbb R\). Inserting the formulae for
\(\alpha_{p,k,n}\) and \(f_{p,k,n}\), and dividing by \(2\), proves
\eqref{eq:shifted-zeta-chebyshev-arcsine} and
\eqref{eq:shifted-beta-chebyshev-arcsine}.
\end{proof}

\begin{corollary}[Odd powers of \(\sinh x\)]
\label[corollary]{cor:odd-sinh-power-shifted-integrals}
Let \(n\in\mathbb N\) and \(0\le k<n\). Then
\begin{equation}
\label{eq:odd-sinh-zeta-general}
\int_0^\infty
\frac{\sinh^{2k+1}x}{x\cosh^{2n+1}x}\,dx
=
\sum_{p=1}^{n}
(2^{2p+1}-1)
[z^{2n}](z^2-1)^k(\arcsin z)^{2p}
\frac{\zeta(2p+1)}{\pi^{2p}},
\end{equation}
\begin{equation}
\label{eq:odd-sinh-beta-general}
\int_0^\infty
\frac{\sinh^{2k+1}x}{x\cosh^{2n}x}\,dx
=
\sum_{p=1}^{n}
2^{2p}[z^{2n-1}](z^2-1)^k(\arcsin z)^{2p-1}
\frac{\beta(2p)}{\pi^{2p-1}}.
\end{equation}
\end{corollary}

\begin{proof}
By the classical odd-power linearization formula; see Equation (4) in Theorem 7.1 of
\cite{HerbigGoncalves2023},
\[
\sinh^{2k+1}x
=
\frac1{4^k}
\sum_{r=0}^{k}
(-1)^{k-r}
\binom{2k+1}{k-r}
\sinh((2r+1)x).
\]
Since \(0\le r\le k<n\), we may apply
\cref{prop:shifted_integrals_chebyshev_arcsine} to each term. Hence
\[
\int_0^\infty
\frac{\sinh^{2k+1}x}{x\cosh^{2n+1}x}\,dx
=
\sum_{p=1}^{n}
(2^{2p+1}-1)
[z^{2n}]
\left(
\frac1{4^k}
\sum_{r=0}^{k}
(-1)^{k-r}
\binom{2k+1}{k-r}
U_{2r}(z)
\right)
(\arcsin z)^{2p}
\frac{\zeta(2p+1)}{\pi^{2p}},
\]
\[
\int_0^\infty
\frac{\sinh^{2k+1}x}{x\cosh^{2n}x}\,dx
=
\sum_{p=1}^{n}
2^{2p}
[z^{2n-1}]
\left(
\frac1{4^k}
\sum_{r=0}^{k}
(-1)^{k-r}
\binom{2k+1}{k-r}
U_{2r}(z)
\right)
(\arcsin z)^{2p-1}
\frac{\beta(2p)}{\pi^{2p-1}}.
\]

It remains to collapse the Chebyshev sum. Since
\[
\sinh((2r+1)x)=\sinh x\,U_{2r}(\cosh x),
\]
the same linearization gives, after division by \(\sinh x\),
\[
\sinh^{2k}x
=
\frac1{4^k}
\sum_{r=0}^{k}
(-1)^{k-r}
\binom{2k+1}{k-r}
U_{2r}(\cosh x).
\]
Putting \(z=\cosh x\), and using
\(\sinh^2x=\cosh^2x-1=z^2-1\), we get the polynomial identity
\[
\frac1{4^k}
\sum_{r=0}^{k}
(-1)^{k-r}
\binom{2k+1}{k-r}
U_{2r}(z)
=
(z^2-1)^k.
\]
Substituting this identity into the two preceding coefficient formulae proves
\eqref{eq:odd-sinh-zeta-general} and \eqref{eq:odd-sinh-beta-general}.
\end{proof}

\begin{corollary}[Integrals involving the reciprocal hyperbolic arctangent]
\label[corollary]{cor:reciprocal-arctanh-integrals}
Let \(n\in\mathbb N\). Then
\begin{equation}
\label{eq:reciprocal-arctanh-zeta}
\int_0^1
\frac{x^{2n-1}}{\operatorname{arctanh}x}\,dx
=
\sum_{p=1}^{n}
\left(2^{2p+1}-1\right)
[z^{2n}]
(z^2-1)^{n-1}(\arcsin z)^{2p}
\frac{\zeta(2p+1)}{\pi^{2p}},
\end{equation}
\begin{equation}
\label{eq:reciprocal-arctanh-beta}
\int_0^1
\frac{x^{2n-1}}
{\sqrt{1-x^2}\,\operatorname{arctanh}x}\,dx
=
\sum_{p=1}^{n}
2^{2p}
[z^{2n-1}]
(z^2-1)^{n-1}(\arcsin z)^{2p-1}
\frac{\beta(2p)}{\pi^{2p-1}}.
\end{equation}
\end{corollary}

\begin{proof}
We make the change of variables
\(
x=\tanh t\) and apply \cref{cor:odd-sinh-power-shifted-integrals} with
\(k=n-1\).
\end{proof}

\section{Coefficient Formulae for the Even-Frequency Shifted Integrals}
\label[section]{sec:coefficient-formulae_double_angle}

\begin{proposition}
\label[proposition]{prop:sinh-over-cosh-integrals}
Let \(m,k\in\mathbb N\) with \(m\geq k\geq1\). Then
\[
\begin{cases}
\displaystyle
\int_{0}^{\infty}
\frac{\sinh(2kx)}{x\cosh^{2m+2}x}\,dx
=
\sum_{p=1}^{m}
\left(2^{2p+1}-1\right)
[z^{2m+1}]
U_{2k-1}(z)(\arcsin z)^{2p}
\frac{\zeta(2p+1)}{\pi^{2p}},
\\[2.2ex]
\displaystyle
\int_{0}^{\infty}
\frac{\sinh(2kx)}{x\cosh^{2m+1}x}\,dx
=
\sum_{p=1}^{m}
2^{2p}
[z^{2m}]
U_{2k-1}(z)(\arcsin z)^{2p-1}
\frac{\beta(2p)}{\pi^{2p-1}}.
\end{cases}
\]
\end{proposition}

\begin{proof}
The result follows directly from
\cref{prop:shifted_integrals_chebyshev_arcsine}.
Indeed,
\[
2\cosh x\,\sinh(2kx)
=
\sinh((2k+1)x)+\sinh((2k-1)x).
\]
Hence
\[
\begin{cases}
\displaystyle
\int_{0}^{\infty}
\frac{\sinh(2kx)}{x\cosh^{2m+2}x}\,dx
=
\frac12
\int_{0}^{\infty}
\frac{\sinh((2k+1)x)+\sinh((2k-1)x)}
{x\cosh^{2m+3}x}\,dx,
\\[2.2ex]
\displaystyle
\int_{0}^{\infty}
\frac{\sinh(2kx)}{x\cosh^{2m+1}x}\,dx
=
\frac12
\int_{0}^{\infty}
\frac{\sinh((2k+1)x)+\sinh((2k-1)x)}
{x\cosh^{2m+2}x}\,dx.
\end{cases}
\]

With \(n=m+1\), the two integrals on the right are one half
of the sums of the shifted integrals in
\cref{prop:shifted_integrals_chebyshev_arcsine}
corresponding to the indices \(k\) and \(k-1\).
Since \(m\geq k\geq1\), both shifted indices are admissible.
Therefore
\[
\begin{cases}
\displaystyle
\int_{0}^{\infty}
\frac{\sinh(2kx)}{x\cosh^{2m+2}x}\,dx
=
\frac12
\sum_{p=1}^{m+1}
\left(2^{2p+1}-1\right)
[z^{2m+2}]
\bigl(U_{2k}(z)+U_{2k-2}(z)\bigr)
(\arcsin z)^{2p}
\frac{\zeta(2p+1)}{\pi^{2p}},
\\[2.2ex]
\displaystyle
\int_{0}^{\infty}
\frac{\sinh(2kx)}{x\cosh^{2m+1}x}\,dx
=
\frac12
\sum_{p=1}^{m+1}
2^{2p}
[z^{2m+1}]
\bigl(U_{2k}(z)+U_{2k-2}(z)\bigr)
(\arcsin z)^{2p-1}
\frac{\beta(2p)}{\pi^{2p-1}}.
\end{cases}
\]

Using the Chebyshev recurrence \(
U_{2k}(z)+U_{2k-2}(z)
=
2zU_{2k-1}(z),
\) we obtain
\[
\begin{cases}
\displaystyle
\int_{0}^{\infty}
\frac{\sinh(2kx)}{x\cosh^{2m+2}x}\,dx
=
\sum_{p=1}^{m+1}
\left(2^{2p+1}-1\right)
[z^{2m+1}]
U_{2k-1}(z)(\arcsin z)^{2p}
\frac{\zeta(2p+1)}{\pi^{2p}},
\\[2.2ex]
\displaystyle
\int_{0}^{\infty}
\frac{\sinh(2kx)}{x\cosh^{2m+1}x}\,dx
=
\sum_{p=1}^{m+1}
2^{2p}
[z^{2m}]
U_{2k-1}(z)(\arcsin z)^{2p-1}
\frac{\beta(2p)}{\pi^{2p-1}}.
\end{cases}
\]

The terms \(p=m+1\) inherited from
\cref{prop:shifted_integrals_chebyshev_arcsine}
vanish by degree arguments, so both sums terminate at \(p=m\). Proving the stated
formulae.
\end{proof}

\section{Coefficient Formulae for the logarithmic tangent integrals $\displaystyle\int_0^{\pi/4}
\frac{\sin(4nx)}{\ln(\tan x)}\,dx$ and $\displaystyle\int_0^{\pi/4}
\frac{\cos((4n-2)x)}{\ln(\tan x)}\,dx$}\label[section]{sec:coefficient-formulae-logarithmic-tangent-integrals}

The corresponding full-interval integrals were already evaluated by
Li and Chu \cite{LiChu2024LogTan}.  In their notation,
\[
\Phi(n)=
\int_0^{\pi/2}\frac{\sin(4nx)}{\ln(\tan x)}\,dx,
\qquad
\Psi(n-1)=
\int_0^{\pi/2}\frac{\cos((4n-2)x)}{\ln(\tan x)}\,dx.
\]
Under \(x\mapsto \pi/2-x\), both quotients are invariant, and hence
\[
\Phi(n)=
2\int_0^{\pi/4}\frac{\sin(4nx)}{\ln(\tan x)}\,dx,
\qquad
\Psi(n-1)=
2\int_0^{\pi/4}\frac{\cos((4n-2)x)}{\ln(\tan x)}\,dx.
\]
Accordingly, the purpose of this section is not a new evaluation of these
integrals, but a compression of their zeta- and beta-coefficients into the
Chebyshev--arcsine Laurent coefficient forms below.

\begin{proposition}[A reciprocal Chebyshev binomial identity]
\label[proposition]{prop:reciprocal-chebyshev-binomial-identity}
Let \(m\in \mathbb N\). Then
\[
\sum_{k=0}^{\lfloor (m-1)/2\rfloor}
(-1)^k
\binom{2m}{m-2k-1}
U_{2k}(X)
=
2^{m-1}X^{m-1}
U_{m-1}\!\left(\frac1X\right).
\]
The identity holds in \(\mathbb Z[X]\).
\end{proposition}

\begin{proof}
Let \(\displaystyle
S_m(X):=
\sum_{k=0}^{\lfloor(m-1)/2\rfloor}
(-1)^k\binom{2m}{m-2k-1}U_{2k}(X).
\) Using the convention that binomial coefficients vanish
outside their natural range and bearing in mind that \(\displaystyle\binom{2m}{m-2k-1}= [z^{m-2k-1}](1+z)^{2m}\), we obtain
\[
\begin{aligned}
S_m(X)
&=\sum_{k=0}^{\infty}(-1)^k U_{2k}(X)
  [z^{m-2k-1}](1+z)^{2m}\\
&=[z^{m-1}](1+z)^{2m}
  \sum_{k=0}^{\infty}(-1)^k U_{2k}(X)z^{2k}.
\end{aligned}
\]

By the generating-function identity for the
even-indexed Chebyshev polynomials of the
second kind
\cite[Proposition~1.1, p.~1660]{BostanSalvyTran2010},
we have
\[
\sum_{k=0}^{\infty}(-1)^k U_{2k}(X)z^{2k}
=
\frac{1-z^2}
{(1-z^2)^2+4X^2z^2}
=
\frac{1-z^2}
{1+(4X^2-2)z^2+z^4}.
\]
Consequently,
\[
S_m(X)=[z^{m-1}]
\frac{(1+z)^{2m}(1-z^2)}
{1+(4X^2-2)z^2+z^4}=\frac{1}{2\pi i}
\oint_\gamma
\frac{(1+z)^{2m}(1-z^2)}
{z^m\bigl(1+(4X^2-2)z^2+z^4\bigr)}
\,dz,
\]
where \(\gamma\) is a sufficiently small positively
oriented circle about the origin.

Now make the substitution \(\displaystyle
t=\frac{z}{(1+z)^2}.
\) Since \(t(0)=0\) and \(t'(0)=1\), this change of
variables is locally biholomorphic. Moreover,\(\displaystyle
dt=\frac{1-z}{(1+z)^3}\,dz
\) and \(
1+(4X^2-2)z^2+z^4
=(1+z)^4(1-4t+4X^2t^2).
\)
Thus, writing \(\Gamma=t(\gamma)\), we obtain
\[
S_m(X)
=\frac{1}{2\pi i}
\oint_\Gamma
\frac{dt}{t^m(1-4t+4X^2t^2)}=[t^{m-1}]
\frac{1}{1-4t+4X^2t^2}.
\]

Finally, the generating function of the Chebyshev
polynomials of the second kind (see \cite[Eq.~(1.5)]{KimEtAl2018}) gives, for \(X\ne0\),
\[
\frac{1}{1-4t+4X^2t^2}
=
\sum_{n=0}^{\infty}
2^nX^nU_n\!\left(\frac1X\right)t^n.
\]
Therefore, \(
S_m(X)=
2^{m-1}X^{m-1}
U_{m-1}\!\left(\dfrac1X\right).
\)
Since both sides are polynomials in \(X\), the
identity extends to \(X=0\) and holds in
\(\mathbb Z[X]\).
\end{proof}

\begin{proposition}[Compact coefficient form of the logarithmic tangent integrals]
\label[proposition]{proposition:logtan-integrals-compact}
Let \(U_m\) denote the Chebyshev polynomial of the second kind.  For every
\(n\in\mathbb N\),
\begin{equation}
\label{eq:sin-logtan-compact}
\int_0^{\pi/4}
\frac{\sin(4nx)}{\ln(\tan x)}\,dx
=
\frac{(-1)^n}{2}
\sum_{p=1}^{n}
(2^{2p+1}-1)
[x]\,
U_{2n-1}\!\left(\frac1x\right)(\arcsin x)^{2p}
\frac{\zeta(2p+1)}{\pi^{2p}},
\end{equation}
\begin{equation}
\label{eq:cos-logtan-compact}
\int_0^{\pi/4}
\frac{\cos((4n-2)x)}{\ln(\tan x)}\,dx
=
(-1)^n
\sum_{p=1}^{n}
2^{2p-1}
[x]\,
U_{2n-2}\!\left(\frac1x\right)(\arcsin x)^{2p-1}
\frac{\beta(2p)}{\pi^{2p-1}}.
\end{equation}
Here \([x]\) denotes the coefficient of \(x\) in the Laurent expansion at
\(x=0\).
\end{proposition}

\begin{proof}
We start from the logarithmic tangent evaluations
\cite{talla_waffo_integral_2025}; see also the full-interval evaluations
of Li and Chu \cite{LiChu2024LogTan}:
\begin{equation}\label{eq:talla-waffo-logtan-start}
\begin{cases}
\displaystyle
\int_0^{\pi/4}
\frac{\sin(4nx)}{\ln(\tan x)}\,dx
=
\sum_{p=1}^{n}
q_{p,n}\frac{\zeta(2p+1)}{\pi^{2p}},
\\[2.2ex]
\displaystyle
\int_0^{\pi/4}
\frac{\cos((4n-2)x)}{\ln(\tan x)}\,dx
=
\sum_{p=1}^{n}
\eta_{p,n}\frac{\beta(2p)}{\pi^{2p-1}},
\end{cases}
\end{equation}
where
\begin{equation}\label{eq:talla-waffo-logtan-coefficients}
\begin{cases}
\displaystyle
q_{p,n}
=
\frac{1}{4^n}
\sum_{k=0}^{n-1}
\binom{4n}{2n-2k-1}
\sum_{m=p}^{n}
(-1)^p c_{2n-2m,n}
\frac{(2^{2p+1}-1)(2k+1)^{2m-2p}}{(2m-2p)!},
\\[2.2ex]
\displaystyle
\eta_{p,n}
=
\frac{(-1)^p}{2^{2n-2p-1}}
\sum_{m=0}^{n-p}
\frac{d_{2m,n}}{(2n-2m-2p)!}
\sum_{k=0}^{n-1}
\binom{4n-2}{2n-2k-2}
(2k+1)^{2n-2m-2p}.
\end{cases}
\end{equation}

By \cref{lemma:cd-sums-to-arcsine-chebyshev}, this gives
\begin{equation}\label{eq:logtan-coefficients-arcsine-step}
\begin{cases}
\displaystyle
q_{p,n}
=
(2^{2p+1}-1)
[x^{2n}]
\left\{
\frac{(-1)^n}{4^n}
\sum_{k=0}^{n-1}
(-1)^k
\binom{4n}{2n-2k-1}
U_{2k}(x)
\right\}
(\arcsin x)^{2p},
\\[2.2ex]
\displaystyle
\eta_{p,n}
=
2^{2p}
[x^{2n-1}]
\left\{
\frac{(-1)^n}{2^{2n-1}}
\sum_{k=0}^{n-1}
(-1)^k
\binom{4n-2}{2n-2k-2}
U_{2k}(x)
\right\}
(\arcsin x)^{2p-1}.
\end{cases}
\end{equation}

Using \cref{prop:reciprocal-chebyshev-binomial-identity} with $m = 2n$ and $m = 2n-1$, respectively, we obtain
\begin{equation}\label{eq:logtan-coefficients-collapsed-step}
\begin{cases}
\displaystyle
q_{p,n}
=
\frac{(-1)^n}{2}
(2^{2p+1}-1)
[x^{2n}]
x^{2n-1}
U_{2n-1}\!\left(\frac1x\right)
(\arcsin x)^{2p},
\\[2.2ex]
\displaystyle
\eta_{p,n}
=
(-1)^n2^{2p-1}
[x^{2n-1}]
x^{2n-2}
U_{2n-2}\!\left(\frac1x\right)
(\arcsin x)^{2p-1}.
\end{cases}
\end{equation}

For every Laurent series \(F(x)\) and every integer \(r\), we use the
coefficient shift
\[
[x^r]\,x^{r-1}F(x)=[x]\,F(x).
\]
Hence
\begin{equation}\label{eq:logtan-coefficients-compact-step}
\begin{cases}
\displaystyle
q_{p,n}
=
\frac{(-1)^n}{2}
(2^{2p+1}-1)
[x]\,
U_{2n-1}\!\left(\frac1x\right)
(\arcsin x)^{2p},
\\[2.2ex]
\displaystyle
\eta_{p,n}
=
(-1)^n2^{2p-1}
[x]\,
U_{2n-2}\!\left(\frac1x\right)
(\arcsin x)^{2p-1}.
\end{cases}
\end{equation}

Substituting \eqref{eq:logtan-coefficients-compact-step} into
\eqref{eq:talla-waffo-logtan-start} proves
\eqref{eq:sin-logtan-compact} and \eqref{eq:cos-logtan-compact}.
\end{proof}

\begin{remark}[The first beta coefficient]
\label[remark]{rem:first-beta-coefficient}
The first three cases of \eqref{eq:cos-logtan-compact} are
\begin{align}
\int_0^{\pi/4}
\frac{\cos(2x)}{\ln(\tan x)}\,dx
&=
-\frac{2\beta(2)}{\pi},
\label{eq:cos2x-logtan}
\\
\int_0^{\pi/4}
\frac{\cos(6x)}{\ln(\tan x)}\,dx
&=
-\frac{2\beta(2)}{3\pi}
+
\frac{32\beta(4)}{\pi^3},
\label{eq:cos6x-logtan}
\\
\int_0^{\pi/4}
\frac{\cos(10x)}{\ln(\tan x)}\,dx
&=
-\frac{2\beta(2)}{5\pi}
+
\frac{32\beta(4)}{\pi^3}
-
\frac{512\beta(6)}{\pi^5}.
\label{eq:cos10x-logtan}
\end{align}
All three evaluations already appear in
\textnormal{\cite{talla_waffo_integral_2025}}.

The first three examples exhibit a particularly simple pattern for the
coefficient of \(\beta(2)/\pi\). Indeed, the term corresponding to
\(p=1\) in \eqref{eq:cos-logtan-compact} is \(\displaystyle
(-1)^n
2[x]\,
U_{2n-2}\!\left(\frac1x\right)\arcsin x
\frac{\beta(2)}{\pi},
\)
and one has
\begin{equation}
[x]\,
U_{2n-2}\!\left(\frac1x\right)\arcsin x
=
\frac{(-1)^{n-1}}{2n-1}.
\end{equation}
Consequently, the coefficient of \(\beta(2)/\pi\) is
\(
-\dfrac{2}{2n-1}.
\)

A proof of this coefficient is already given
in \textnormal{\cite{talla_waffo_integral_2025}}. In
\cref{prop:first-chebyshev-arcsin-coefficient}, we provide a different
proof based directly on the new coefficient-extraction formula
\eqref{eq:cos-logtan-compact}.
\end{remark}

\begin{proposition}[The first Chebyshev--arcsine coefficient]
\label[proposition]{prop:first-chebyshev-arcsin-coefficient}
For every integer \(n\geq1\),
\begin{equation}
\label{eq:first-chebyshev-arcsin-coefficient}
\boxed{
[x]\,
U_{2n-2}\!\left(\frac1x\right)\arcsin x
=
\frac{(-1)^{n-1}}{2n-1}
}.
\end{equation}
\end{proposition}

\begin{proof}
Put $m=n-1$. The standard expansions
\[
  U_{2m}\!\left(\frac{1}{x}\right)
  =\sum_{r=0}^{m}(-1)^{m-r}
    \binom{m+r}{m-r}4^r x^{-2r},
\qquad
\textnormal{and}
\qquad
  \arcsin x
  =\sum_{r=0}^{\infty}
    \frac{\binom{2r}{r}}{4^r(2r+1)}x^{2r+1}
\]
show that only matching indices contribute to the
coefficient of $x$. Using
\[
  \binom{m+r}{m-r}\binom{2r}{r}
  =\binom mr\binom{m+r}{r},
\]
we obtain
\begin{equation}\label{eq:coefficient-sum}
  [x] \, U_{2m}\!\left(\frac{1}{x}\right)\arcsin x
  =(-1)^m\sum_{r=0}^{m}
    \frac{(-1)^r}{2r+1}
    \binom mr\binom{m+r}{r}.
\end{equation}

By Murphy's representation of the Legendre
polynomials \cite[Eq.~(1.3)]{sun2013},
\[
  P_m(1-2t^2)
  =\sum_{r=0}^{m}(-1)^r
    \binom mr\binom{m+r}{r}t^{2r}.
\]
Since \(\displaystyle
  \frac{1}{2r+1}=\int_0^1 t^{2r}\,dt,
\) \eqref{eq:coefficient-sum} becomes
\begin{equation}\label{eq:legendre-integral}
  [x] \, U_{2m}\!\left(\frac{1}{x}\right)\arcsin x
  =(-1)^m\int_0^1 P_m(1-2t^2)\,dt.
\end{equation}

The classical integral formula
\cite[Eq ~7.233]{gradshteyn2014} reads
\[
  \int_0^1 t^{2\mu-1}P_m(1-2t^2)\,dt
  =\frac{(-1)^m\Gamma(\mu)^2}
         {2\Gamma(\mu+m+1)\Gamma(\mu-m)},
  \qquad \operatorname{Re}(\mu)>0.
\]
Taking $\mu=\tfrac12$ and using the reflection
and recurrence formulas for the Gamma function
yields
\[
  \int_0^1 P_m(1-2t^2)\,dt
  =\frac{1}{2m+1}.
\]
Substituting this into \eqref{eq:legendre-integral}
and recalling $m=n-1$ proves the claim.
\end{proof}

\begin{remark}[A parity-free formulation]
\label[remark]{rem:parity-free-terminal-coefficient}
The terminal coefficients in the sine and cosine expansions were already
determined in \textnormal{\cite{talla_waffo_integral_2025}}. More precisely, one has
\begin{equation}
\label{eq:previous-terminal-coefficients}
q_{n,n}
=
(-1)^n
\bigl(2^{2n+1}-1\bigr)4^{n-1},
\qquad
\eta_{n,n}
=
(-1)^n2^{4n-3}.
\end{equation}

We now compare these formulae with the coefficient-extraction identities
\eqref{eq:sin-logtan-compact} and \eqref{eq:cos-logtan-compact}.
The coefficient of \(\zeta(2n+1)/\pi^{2n}\) in
\eqref{eq:sin-logtan-compact}, obtained by taking \(p=n\), is \(\displaystyle
\frac{(-1)^n}{2}
\bigl(2^{2n+1}-1\bigr)
[x]\,
U_{2n-1}\!\left(\frac1x\right)
(\arcsin x)^{2n}.
\) Comparing this with \(q_{n,n}\) in
\eqref{eq:previous-terminal-coefficients} gives
\begin{equation}
\label{eq:terminal-sine-coefficient-extraction}
[x]\,
U_{2n-1}\!\left(\frac1x\right)
(\arcsin x)^{2n}
=
2^{2n-1}.
\end{equation}

Similarly, the coefficient of \(\beta(2n)/\pi^{2n-1}\) in
\eqref{eq:cos-logtan-compact} is \(\displaystyle
(-1)^n2^{2n-1}
[x]\,
U_{2n-2}\!\left(\frac1x\right)
(\arcsin x)^{2n-1}.
\) Comparing this with \(\eta_{n,n}\) in
\eqref{eq:previous-terminal-coefficients} yields
\begin{equation}
\label{eq:terminal-cosine-coefficient-extraction}
[x]\,
U_{2n-2}\!\left(\frac1x\right)
(\arcsin x)^{2n-1}
=
2^{2n-2}.
\end{equation}

The two identities
\eqref{eq:terminal-sine-coefficient-extraction} and
\eqref{eq:terminal-cosine-coefficient-extraction} admit a single
formulation independent of parity. Indeed, setting \(r=2n\) in the
first and \(r=2n-1\) in the second gives
\begin{equation}
\label{eq:parity-free-terminal-coefficient}
[x]\,
U_{r-1}\!\left(\frac1x\right)
(\arcsin x)^r
=
2^{r-1},
\qquad r\geq1.
\end{equation}

Thus the terminal coefficients in both
\eqref{eq:sin-logtan-compact} and \eqref{eq:cos-logtan-compact}
are governed by the same parity-independent coefficient-extraction
identity. Although \eqref{eq:parity-free-terminal-coefficient} follows
here from the terminal coefficient formulae established in
\textnormal{\cite{talla_waffo_integral_2025}}, in
\cref{prop:parity-free-terminal-coefficient} we give a new direct proof
based solely on the coefficient-extraction representation. This single
argument proves simultaneously
\eqref{eq:terminal-sine-coefficient-extraction} and
\eqref{eq:terminal-cosine-coefficient-extraction}.
\end{remark}

\begin{proposition}[The parity-free terminal coefficient]
\label[proposition]{prop:parity-free-terminal-coefficient}
For every integer \(r\geq1\),
\begin{equation}
\label{eq:parity-free-terminal-coefficient-proposition}
\boxed{
[x]\,
U_{r-1}\!\left(\frac1x\right)
(\arcsin x)^r
=
2^{r-1}
}.
\end{equation}
\end{proposition}

\begin{proof}
The Chebyshev polynomial \(U_{r-1}\) has degree \(r-1\), leading
coefficient \(2^{r-1}\), and contains only powers of the same parity as
\(r-1\). Therefore, as a Laurent expansion at \(x=0\),
\[
U_{r-1}\!\left(\frac1x\right)
=
2^{r-1}x^{1-r}+O(x^{3-r}).
\]

Moreover, \(
\arcsin x=x+O(x^3),
\)
and hence \(
(\arcsin x)^r
=
x^r+O(x^{r+2}).
\)
Multiplying these expansions gives
\[
U_{r-1}\!\left(\frac1x\right)(\arcsin x)^r
=
\left(2^{r-1}x^{1-r}+O(x^{3-r})\right)
\left(x^r+O(x^{r+2})\right)=
2^{r-1}x+O(x^3).
\]
In particular, no other term can contribute to the coefficient of \(x\).
Therefore \(\displaystyle
[x]\,
U_{r-1}\!\left(\frac1x\right)(\arcsin x)^r
=
2^{r-1},
\) which proves
\eqref{eq:parity-free-terminal-coefficient-proposition}.
\end{proof}

\section{The tangent-power integrals on the half-line}\label[section]{section:general-tangent-integral}

The two tangent-power integrals considered in this section are kept in the
same parameters \(m,n\) as in their individual formulations.  The auxiliary
polynomial device is more naturally written with independent parameters
\(l,r,q\), so that the two cases arise from different specializations of one
family.

\begin{lemma}[Derivative-polynomial family]
\label[lemma]{lem:unified-derivative-polynomial}
Let \(l\geq1\) and \(q\geq0\) be integers.  For every integer \(r\geq0\),
there is a unique polynomial \(P^{l,q}_r\) such that
\begin{equation}
\label{eq:unified-derivative-polynomial-definition}
\frac{d^r}{dx^r}\left(\frac{\tanh^l x}{\cosh^q x}\right)
=
\frac{1}{\cosh^q x}P^{l,q}_r(\tanh x).
\end{equation}
It satisfies
\begin{equation}
\label{eq:unified-derivative-polynomial-recurrence}
P^{l,q}_0(t)=t^l,
\qquad
P^{l,q}_{r+1}(t)
=(1-t^2)\bigl(P^{l,q}_r\bigr)'(t)-qtP^{l,q}_r(t).
\end{equation}
If \(0\leq r\leq l\), then only the powers
\(
t^{l-r},t^{l-r+2},\ldots,t^{l+r}
\)
can occur.  In particular, if \(l-r\) is odd, then
\(P^{l,q}_r\) is odd.  Moreover,
\begin{equation}
\label{eq:unified-polynomial-value-at-one}
P^{l,q}_r(1)=(-q)^r.
\end{equation}
Consequently, if \(q=0\), \(r\geq1\), and \(l-r\) is odd, then
\(P^{l,0}_r(t)\) is divisible by \(t(1-t^2)\).
\end{lemma}

\begin{proof}
Put \(t=\tanh x\).  Since \(dt/dx=1-t^2\) and
\[
\frac{d}{dx}(\cosh^{-q}x)=-qt\cosh^{-q}x,
\]
differentiating \(P(t)/\cosh^q x\) gives
\[
\frac{(1-t^2)P'(t)-qtP(t)}{\cosh^q x},
\]
which proves the recurrence.  Applied to a monomial \(t^k\), the recurrence
produces only \(t^{k-1}\) and \(t^{k+1}\); induction gives the stated
support and parity.  Evaluating the recurrence at \(t=1\) yields
\(P^{l,q}_{r+1}(1)=-qP^{l,q}_r(1)\), proving
\eqref{eq:unified-polynomial-value-at-one}.  The final divisibility statement
follows from oddness together with the zeros at \(t=\pm1\).
\end{proof}

The following coefficient identity is the common collapse needed for both
integrals.

\begin{lemma}[General coefficient collapse]
\label[lemma]{lem:unified-tangent-collapse}
Let \(0\leq r\leq l\), let \(q\geq0\) be an integer, and suppose that
\(l-r\) is odd.  Write
\[
P^{l,q}_r(t)
=
\sum_{k=(l-r-1)/2}^{(l+r-1)/2}
c_{l,r,q,k}t^{2k+1}.
\]
For \(s \in \mathbb{N}_0\), define
\[
C_{l,r,q}(s)
:=
\sum_{k=(l-r-1)/2}^{(l+r-1)/2}
c_{l,r,q,k}
[z^{2k+q}](z^2-1)^k(\arcsin z)^s.
\]
Then
\begin{equation}
\label{eq:unified-tangent-collapse}
C_{l,r,q}(s)
=
\begin{cases}
0, & s<r,\\[2mm]
\displaystyle
(-1)^{(l-r-1)/2}
\frac{s!}{(s-r)!}
[u^{l+q+r-s-1}]
\left(\frac{u}{\sin u}\right)^q
(u\cot u)^l,
& s\geq r.
\end{cases}
\end{equation}
\end{lemma}

\begin{proof}
Since \(P^{l,q}_r\) is odd, write
\[
P^{l,q}_r(t)=tA(t^2),
\qquad
A(w)=
\sum_{k=(l-r-1)/2}^{(l+r-1)/2}
c_{l,r,q,k}w^k.
\]
Because \((z^2-1)^k=z^{2k}(1-z^{-2})^k\),
\[
C_{l,r,q}(s)
=
[z^q]A(1-z^{-2})(\arcsin z)^s.
\]
Cauchy's coefficient formula followed by \(z=\sin u\) gives
\[
C_{l,r,q}(s)
=
\frac1{i}\frac1{2\pi i}
\oint
u^s\frac{P^{l,q}_r(i\cot u)}{\sin^q u}\,du.
\]
Now set \(x=i(\pi/2-u)\).  Then
\[
\tanh x=i\cot u,
\qquad
\cosh x=\sin u,
\qquad
\frac{d}{dx}=i\frac{d}{du},
\]
and \eqref{eq:unified-derivative-polynomial-definition} yields
\[
\frac{P^{l,q}_r(i\cot u)}{\sin^q u}
=
i^{l+r}
\frac{d^r}{du^r}
\left(\frac{\cos^l u}{\sin^{l+q}u}\right).
\]
Hence
\[
C_{l,r,q}(s)
=
i^{l+r-1}\frac1{2\pi i}
\oint
u^s
\frac{d^r}{du^r}
\left(\frac{\cos^l u}{\sin^{l+q}u}\right)du.
\]
Integrating by parts \(r\) times around the closed contour gives zero for
\(s<r\).  If \(s\geq r\), it gives
\[
C_{l,r,q}(s)
=
i^{l+r-1}(-1)^r\frac{s!}{(s-r)!}
\frac1{2\pi i}
\oint
u^{s-r}\frac{\cos^l u}{\sin^{l+q}u}\,du.
\]
Finally,
\[
\frac{\cos^l u}{\sin^{l+q}u}
=
u^{-l-q}
\left(\frac{u}{\sin u}\right)^q
(u\cot u)^l,
\]
and, since \(l-r\) is odd,
\[
i^{l+r-1}(-1)^r=(-1)^{(l-r-1)/2}.
\]
Cauchy's coefficient formula now gives
\eqref{eq:unified-tangent-collapse}.
\end{proof}

Li and Chu's integral \(J(M,N)\) in \cite{LiChu2024Tanh} coincides
with the zeta-type family below after the index shift
\[
M=m+1,\qquad N=n+1.
\]
Their Theorem~6.1 gives a contour-integral evaluation whose coefficients are
expressed through nested finite sums.  The first formula in the following
proposition is therefore best viewed as a compact coefficient-extraction
reformulation of the same integral family.  The second formula is its
Dirichlet-beta companion obtained from the unified derivative-polynomial
mechanism developed above.

\begin{proposition}[Zeta and Dirichlet beta formulae for the hyperbolic tangent power integrals]
\label[proposition]{prop:general-tangent-power-integral}
Let \(m,n\in \mathbb N\), \(m\geq n\), and suppose that \(m+n\) is even.  Then
\begin{equation}
\label{eq:general-tangent-power-integral-formula}
\boxed{
\int_0^\infty
\frac{\tanh^{m+1}x}{x^{n+1}}\,dx
=
(-1)^{(m-n)/2}
\sum_{p=\lceil n/2\rceil}^{(m+n)/2}
\binom{2p}{n}
(2^{2p+1}-1)
\frac{\zeta(2p+1)}{\pi^{2p}}
[u^{m+n-2p}](u\cot u)^{m+1}
}
\end{equation}
\begin{equation}
\label{eq:general-sech-tangent-power-integral-formula}
\boxed{
\int_0^\infty
\frac{\tanh^m x}{x^n\cosh x}\,dx
=
(-1)^{(m-n)/2}
\sum_{p=\lceil n/2\rceil}^{(m+n)/2}
2^{2p}\binom{2p-1}{n-1}
\frac{\beta(2p)}{\pi^{2p-1}}
[u^{m+n-2p}]
\frac{u}{\sin u}(u\cot u)^m
}.
\end{equation}
\end{proposition}

\begin{proof}
For \(\varepsilon\in\{0,1\}\), set
\(l_\varepsilon:=m+1-\varepsilon\),
\(r_\varepsilon:=n-\varepsilon\), and
\[
F_\varepsilon(x)
:=
\frac{\tanh^{l_\varepsilon}x}{\cosh^\varepsilon x},
\qquad
I_\varepsilon
:=
\int_0^\infty
\frac{F_\varepsilon(x)}{x^{r_\varepsilon+1}}\,dx.
\]
Thus \(I_0\) and \(I_1\) are the zeta and Dirichlet-beta integrals,
respectively.

Repeated integration by parts \(r_\varepsilon\) times gives
\[
I_\varepsilon
=
\frac{1}{r_\varepsilon!}
\int_0^\infty
\frac{F_\varepsilon^{(r_\varepsilon)}(x)}{x}\,dx.
\]
Indeed, near \(0\),
\(F_\varepsilon^{(j)}(x)=O(x^{l_\varepsilon-j})\), so every boundary
term is \(O(x^{l_\varepsilon-r_\varepsilon})
=O(x^{m+1-n})\). At infinity, all derivatives decay exponentially
for \(\varepsilon=1\); for \(\varepsilon=0\), the positive-order
derivatives decay exponentially, while the remaining boundary term is
\(O(x^{-n})\).

By \cref{lem:unified-derivative-polynomial},
\[
F_\varepsilon^{(r_\varepsilon)}(x)
=
\frac{1}{\cosh^\varepsilon x}
P^{l_\varepsilon,\varepsilon}_{r_\varepsilon}(\tanh x).
\]
Since \(l_\varepsilon-r_\varepsilon=m+1-n\) is odd,
\(P^{l_\varepsilon,\varepsilon}_{r_\varepsilon}\) is odd; for
\(\varepsilon=0\), it is moreover divisible by \(1-t^2\). Together
with the support given by
\cref{lem:unified-derivative-polynomial}, this yields
\[
P^{l_\varepsilon,\varepsilon}_{r_\varepsilon}(t)
=
t(1-t^2)^{1-\varepsilon}
\sum_{k=(m-n)/2}^{(m+n-2)/2}
a_{\varepsilon,k}t^{2k}.
\]
Hence
\[
I_\varepsilon
=
\frac{1}{(n-\varepsilon)!}
\sum_{k=(m-n)/2}^{(m+n-2)/2}
a_{\varepsilon,k}
\int_0^\infty
\frac{\sinh^{2k+1}x}
{x\cosh^{2k+3-\varepsilon}x}\,dx.
\]

Define
\[
\Lambda_{\varepsilon,p}
:=
\begin{cases}
\displaystyle
(2^{2p+1}-1)\frac{\zeta(2p+1)}{\pi^{2p}},
& \varepsilon=0,\\[2mm]
\displaystyle
2^{2p}\frac{\beta(2p)}{\pi^{2p-1}},
& \varepsilon=1.
\end{cases}
\]
Then \eqref{eq:odd-sinh-zeta-general} and
\eqref{eq:odd-sinh-beta-general}, followed by an interchange of the
finite sums, give
\[
I_\varepsilon
=
\frac{1}{(n-\varepsilon)!}
\sum_{p=1}^{(m+n)/2}
\Lambda_{\varepsilon,p}S_{\varepsilon,p},
\]
where
\[
S_{\varepsilon,p}
=
\sum_{k=(m-n)/2}^{(m+n-2)/2}
a_{\varepsilon,k}
[z^{2k+2-\varepsilon}]
(z^2-1)^k(\arcsin z)^{2p-\varepsilon}.
\]
The terms with \(k<p-1\) vanish automatically, since
\((\arcsin z)^{2p-\varepsilon}=O(z^{2p-\varepsilon})\).

Now let
\(A_\varepsilon(w)
:=\displaystyle\sum_{k=(m-n)/2}^{(m+n-2)/2}a_{\varepsilon,k}w^k\) and
\(B_\varepsilon(w)
:=(1-w)^{1-\varepsilon}A_\varepsilon(w)\). Then
\(P^{l_\varepsilon,\varepsilon}_{r_\varepsilon}(t)
=tB_\varepsilon(t^2)\). Since
\((z^2-1)^k=z^{2k}(1-z^{-2})^k\),
\[
S_{\varepsilon,p}
=
[z^{2-\varepsilon}]
A_\varepsilon(1-z^{-2})
(\arcsin z)^{2p-\varepsilon}.
\]
On the other hand, by the definition of \(C_{l,r,q}\),
\[
C_{l_\varepsilon,r_\varepsilon,\varepsilon}(2p-\varepsilon)
=
[z^\varepsilon]
B_\varepsilon(1-z^{-2})
(\arcsin z)^{2p-\varepsilon}.
\]
Since
\(B_\varepsilon(1-z^{-2})
=z^{-2(1-\varepsilon)}A_\varepsilon(1-z^{-2})\), it follows that
\[
S_{\varepsilon,p}
=
C_{l_\varepsilon,r_\varepsilon,\varepsilon}(2p-\varepsilon).
\]

By \cref{lem:unified-tangent-collapse}, \(S_{\varepsilon,p}=0\)
for \(2p<n\), whereas for \(2p\ge n\),
\[
S_{\varepsilon,p}
=
(-1)^{(m-n)/2}
\frac{(2p-\varepsilon)!}{(2p-n)!}
[u^{m+n-2p}]
\left(\frac{u}{\sin u}\right)^\varepsilon
(u\cot u)^{m+1-\varepsilon}.
\]
The generating function is holomorphic at \(u=0\), so the coefficient
extraction also vanishes for \(2p>m+n\). Thus only
\(\lceil n/2\rceil\le p\le(m+n)/2\) contributes. Finally,
\[
\frac{(2p-\varepsilon)!}
{(n-\varepsilon)!(2p-n)!}
=
\binom{2p-\varepsilon}{n-\varepsilon},
\]
and therefore
\[
I_\varepsilon
=
(-1)^{(m-n)/2}
\sum_{p=\lceil n/2\rceil}^{(m+n)/2}
\binom{2p-\varepsilon}{n-\varepsilon}
\Lambda_{\varepsilon,p}
[u^{m+n-2p}]
\left(\frac{u}{\sin u}\right)^\varepsilon
(u\cot u)^{m+1-\varepsilon}.
\]
Taking \(\varepsilon=0\) and \(\varepsilon=1\) gives
\eqref{eq:general-tangent-power-integral-formula} and
\eqref{eq:general-sech-tangent-power-integral-formula}, respectively.
\end{proof}

\begin{example}\label[example]{example_m_7}
We illustrate \cref{prop:general-tangent-power-integral}. For \(m=7\), one obtains
\begin{align*}
\int_0^\infty \frac{\tanh^8 x}{x^2}\,dx
&=
\frac{352}{15}\frac{\zeta(3)}{\pi^2}
-\frac{5\,456}{15}\frac{\zeta(5)}{\pi^4}
+2\,032\frac{\zeta(7)}{\pi^6}
-4\,088\frac{\zeta(9)}{\pi^8},
\\
\int_0^\infty \frac{\tanh^8 x}{x^4}\,dx
&=
-\frac{21\,824}{105}\frac{\zeta(5)}{\pi^4}
+\frac{22\,352}{3}\frac{\zeta(7)}{\pi^6}
-\frac{228\,928}{3}\frac{\zeta(9)}{\pi^8}
+245\,640\frac{\zeta(11)}{\pi^{10}},
\\
\int_0^\infty \frac{\tanh^8 x}{x^6}\,dx
&=
\frac{44\,704}{35}\frac{\zeta(7)}{\pi^6}
-\frac{1\,259\,104}{15}\frac{\zeta(9)}{\pi^8}
+1\,375\,584\frac{\zeta(11)}{\pi^{10}}
-6\,487\,272\frac{\zeta(13)}{\pi^{12}},
\\
\int_0^\infty \frac{\tanh^8 x}{x^8}\,dx
&=
-\frac{102\,784}{15}\frac{\zeta(9)}{\pi^8}
+720\,544\frac{\zeta(11)}{\pi^{10}}
-17\,299\,392\frac{\zeta(13)}{\pi^{12}}
+112\,456\,344\frac{\zeta(15)}{\pi^{14}}.
\end{align*}
In particular, as \(n\) increases, the first zeta values disappear:
for \(n=3\) there is no \(\zeta(3)\)-term, for \(n=5\) there are no
\(\zeta(3)\)- and \(\zeta(5)\)-terms, and for \(n=7\) the
\(\zeta(3)\)-, \(\zeta(5)\)-, and \(\zeta(7)\)-terms are absent.
\end{example}

\begin{example}\label[example]{example_m_8}
For \(m=8\), the same formula gives
\begin{align*}
\int_0^\infty \frac{\tanh^9 x}{x^3}\,dx
&=
-7\frac{\zeta(3)}{\pi^2}
+\frac{50\,716}{105}\frac{\zeta(5)}{\pi^4}
-7\,239\frac{\zeta(7)}{\pi^6}
+42\,924\frac{\zeta(9)}{\pi^8}
-92\,115\frac{\zeta(11)}{\pi^{10}},
\\
\int_0^\infty \frac{\tanh^9 x}{x^5}\,dx
&=
31\frac{\zeta(5)}{\pi^4}
-\frac{103\,886}{21}\frac{\zeta(7)}{\pi^6}
+135\,926\frac{\zeta(9)}{\pi^8}
-1\,289\,610\frac{\zeta(11)}{\pi^{10}}
+4\,054\,545\frac{\zeta(13)}{\pi^{12}},
\\
\int_0^\infty \frac{\tanh^9 x}{x^7}\,dx
&=
-127\frac{\zeta(7)}{\pi^6}
+\frac{1\,671\,992}{45}\frac{\zeta(9)}{\pi^8}
-1\,633\,506\frac{\zeta(11)}{\pi^{10}}
+22\,705\,452\frac{\zeta(13)}{\pi^{12}}
-98\,399\,301\frac{\zeta(15)}{\pi^{14}},
\\
\int_0^\infty \frac{\tanh^9 x}{x^9}\,dx
&=
511\frac{\zeta(9)}{\pi^8}
-\frac{1\,674\,446}{7}\frac{\zeta(11)}{\pi^{10}}
+15\,407\,271\frac{\zeta(13)}{\pi^{12}}
-295\,197\,903\frac{\zeta(15)}{\pi^{14}}
+1\,686\,883\,770\frac{\zeta(17)}{\pi^{16}}.
\end{align*}
Again the lower zeta values are absent exactly as predicted by the lower
summation bound \(p=\lceil n/2\rceil\), or equivalently by the vanishing
of \(\displaystyle\binom{2p}{n}\) for \(2p<n\).
\end{example}

\begin{remark}
With the change of notation \(a=m+1\) and \(b=n+1\), the integral in
\cref{prop:general-tangent-power-integral} belongs to the family
\[
K(a,b):=\int_0^\infty \frac{\tanh^a x}{x^b}\,dx,
\qquad a\geq b\geq 2,\qquad a\equiv b \pmod 2.
\]
This family is known to admit closed forms in terms of odd zeta values, often
written in umbral notation involving Mittag--Leffler polynomials.%
\footnote{For background discussions of this integral family, see
\url{https://mathoverflow.net/questions/271526/is-there-a-closed-form-for-int-0-infty-frac-tanh3xx2dx}
and \url{https://en.wikipedia.org/wiki/Mittag-Leffler_polynomials\#Closed_forms_of_integral_families}.}
The point of formula~\eqref{eq:general-tangent-power-integral-formula} is that it
gives an explicit coefficient form, with the coefficients extracted directly from
\((u\cot u)^{m+1}\), rather than using umbral notation. For an analogue formula in umbral notation for the integral $\displaystyle\int_0^\infty \frac{\tanh^a x}{x^b \cosh x}\,dx$, see \textnormal{\cite{TallaWaffo2026Umbral}}.
\end{remark}

\begin{remark}[The last zeta and beta coefficients]
The terminal summand in both formulae of
\cref{prop:general-tangent-power-integral} corresponds to
\(p=(m+n)/2\).  Hence \(m+n-2p=0\), so the relevant constant terms are
\([u^0](u\cot u)^{m+1}=1\) and
\([u^0]\dfrac{u}{\sin u}(u\cot u)^m=1\), respectively.
Thus the highest zeta term in
\eqref{eq:general-tangent-power-integral-formula} is
\[
(-1)^{(m-n)/2}\binom{m+n}{n}(2^{m+n+1}-1)
\frac{\zeta(m+n+1)}{\pi^{m+n}},
\]
while the highest Dirichlet-beta term in
\eqref{eq:general-sech-tangent-power-integral-formula} is
\[
(-1)^{(m-n)/2}2^{m+n}\binom{m+n-1}{n-1}
\frac{\beta(m+n)}{\pi^{m+n-1}}.
\]
Consequently, the numerical coefficient multiplying the terminal normalized
zeta value \(\zeta(m+n+1)/\pi^{m+n}\) is an integer, and the same is true
for the numerical coefficient multiplying the terminal normalized beta value
\(\beta(m+n)/\pi^{m+n-1}\).  In particular, the integrality of the last
coefficient is a common structural feature of both tangent-power families.
\end{remark}

\begin{remark}[The first zeta coefficient for even \(n\)] Assume that \(n\) is even. Since \(m+n\) is even, \(m\) is even as well. The first summand in \eqref{eq:general-tangent-power-integral-formula} corresponds to \(p=n/2\). For this value, \[ \binom{2p}{n}=\binom{n}{n}=1 \qquad \textnormal{and} \qquad [u^{m+n-2p}](u\cot u)^{m+1} = [u^m](u\cot u)^{m+1}. \] By the Lagrange inversion  formula \textnormal{\cite{SuryaWarnke2023},\cite[Theorem~5.4.2]{Stanley2023}; see also \cite{Gessel2016}}, applied to the compositional inverse pair \( u=\arctan t, \qquad t=\tan u, \) we obtain \[ [u^m]\left(\frac{u}{\tan u}\right)^{m+1} = (m+1)[t^{m+1}]\arctan t. \] Since \(m\) is even and \(\displaystyle [t^{m+1}]\arctan t = \frac{(-1)^{m/2}}{m+1}, \) it follows that \(\displaystyle [u^m](u\cot u)^{m+1} = (-1)^{m/2}. \) 
Consequently, the first term in \eqref{eq:general-tangent-power-integral-formula} is \(\displaystyle (-1)^{(m-n)/2}(-1)^{m/2} (2^{n+1}-1)\frac{\zeta(n+1)}{\pi^n}. \) Because \(m\) is even, \[ (-1)^{(m-n)/2+m/2} = (-1)^{m-n/2} = (-1)^{n/2}, \] and hence the first zeta term is \(\displaystyle (-1)^{n/2} (2^{n+1}-1)\frac{\zeta(n+1)}{\pi^n}. \) In particular, for even \(n\), the coefficient of the first zeta value is independent of \(m\). \end{remark}

\begin{remark}[Alternation of the relevant coefficients] The alternating signs in the zeta expansion are governed by the coefficients of \((u\cot u)^{m+1}\) occurring in \eqref{eq:general-tangent-power-integral-formula}. These are precisely \[ [u^{2j}](u\cot u)^{m+1}, \qquad 0\leq 2j\leq m. \] Indeed, in \eqref{eq:general-tangent-power-integral-formula}, \( 2j=m+n-2p, \) and the range of \(p\) implies \(0\leq 2j\leq m\). By the Lagrange inversion formula \textnormal{\cite{SuryaWarnke2023}; see also\cite[Theorem~5.4.2]{Stanley2023} and \cite{Gessel2016,Comtet1974}}, applied to the compositional inverse pair \( u=\arctan t, \qquad t=\tan u, \) one obtains, throughout this range, \[ [u^{2j}](u\cot u)^{m+1} = \frac{m+1}{m+1-2j} [t^{2j}] \left(\frac{\arctan t}{t}\right)^{m+1-2j}. \] The exponent \( q:=m+1-2j \) is a positive integer. To determine the sign of the coefficient, use the integral representation \[ \frac{\arctan t}{t} = \int_0^1\frac{dx}{1+t^2x^2}. \] It follows that \[ \left(\frac{\arctan t}{t}\right)^q = \int_{[0,1]^q} \prod_{\ell=1}^q \frac{dx_1\cdots dx_q}{1+t^2x_\ell^2}. \] Expanding each factor as a geometric series gives \[ [t^{2j}] \left(\frac{\arctan t}{t}\right)^q = (-1)^j \int_{[0,1]^q} h_j(x_1^2,\ldots,x_q^2)\, dx_1\cdots dx_q, \] where \(h_j\) denotes the complete homogeneous symmetric polynomial of degree \(j\). Since \[ h_j(x_1^2,\ldots,x_q^2)>0 \] on a set of positive measure in \([0,1]^q\), the integral is strictly positive. Therefore \[ \operatorname{sgn} \left( [t^{2j}] \left(\frac{\arctan t}{t}\right)^q \right) = (-1)^j, \] and hence \[ \operatorname{sgn} \left( [u^{2j}](u\cot u)^{m+1} \right) = (-1)^j. \] Since \( j=\frac{m+n}{2}-p, \) increasing \(p\) by one changes the sign of the corresponding coefficient. All remaining factors in \eqref{eq:general-tangent-power-integral-formula}, apart from the fixed global sign \((-1)^{(m-n)/2}\), are positive. Consequently, the nonzero zeta coefficients alternate in sign. No assertion is made about the coefficients of \((u\cot u)^{m+1}\) outside the range \(0\leq 2j\leq m\). \end{remark}

\medskip
The Dirichlet-beta analogue has equally explicit examples.

\begin{example}
\label{ex:m8-sech-tanh}
For \(m=8\), the coefficient kernel is
\[
\frac{u}{\sin u}(u\cot u)^8
=
1-\frac52u^2+\frac{301}{120}u^4
-\frac{1\,247}{1\,008}u^6+\frac{35}{128}u^8+O(u^{10}).
\]
Thus, for \(n\in\{2,4,6,8\}\), \cref{prop:general-tangent-power-integral} yields
\[
\int_0^\infty
\frac{\tanh^8x}{x^2\cosh x}\,dx
=
-\frac{35}{32}\frac{\beta(2)}{\pi}
+\frac{1\,247}{21}\frac{\beta(4)}{\pi^3}
-\frac{2\,408}{3}\frac{\beta(6)}{\pi^5}
+4\,480\frac{\beta(8)}{\pi^7}
-9\,216\frac{\beta(10)}{\pi^9}.
\]
\[
\int_0^\infty
\frac{\tanh^8x}{x^4\cosh x}\,dx
=
\frac{35}{8}\frac{\beta(4)}{\pi^3}
-\frac{49\,880}{63}\frac{\beta(6)}{\pi^5}
+\frac{67\,424}{3}\frac{\beta(8)}{\pi^7}
-215\,040\frac{\beta(10)}{\pi^9}
+675\,840\frac{\beta(12)}{\pi^{11}}.
\]
\[
\int_0^\infty
\frac{\tanh^8x}{x^6\cosh x}\,dx
=
-\frac{35}{2}\frac{\beta(6)}{\pi^5}
+\frac{19\,952}{3}\frac{\beta(8)}{\pi^7}
-\frac{1\,618\,176}{5}\frac{\beta(10)}{\pi^9}
+4\,730\,880\frac{\beta(12)}{\pi^{11}}
-21\,086\,208\frac{\beta(14)}{\pi^{13}}.
\]
\[
\int_0^\infty
\frac{\tanh^8x}{x^8\cosh x}\,dx
=
70\frac{\beta(8)}{\pi^7}
-\frac{319\,232}{7}\frac{\beta(10)}{\pi^9}
+3\,390\,464\frac{\beta(12)}{\pi^{11}}
-70\,287\,360\frac{\beta(14)}{\pi^{13}}
+421\,724\,160\frac{\beta(16)}{\pi^{15}}.
\]
\end{example}

\begin{example}
\label{ex:m9-sech-tanh}
For \(m=9\), the coefficient kernel is
\[
\frac{u}{\sin u}(u\cot u)^9
=
1-\frac{17}{6}u^2+\frac{239}{72}u^4
-\frac{30\,539}{15\,120}u^6
+\frac{25\,609}{40\,320}u^8+O(u^{10}).
\]
Thus, for \(n\in\{1,3,5,7,9\}\), \cref{prop:general-tangent-power-integral} yields
\[
\int_0^\infty
\frac{\tanh^9x}{x\cosh x}\,dx
=
\frac{25\,609}{10\,080}\frac{\beta(2)}{\pi}
-\frac{30\,539}{945}\frac{\beta(4)}{\pi^3}
+\frac{1\,912}{9}\frac{\beta(6)}{\pi^5}
-\frac{2\,176}{3}\frac{\beta(8)}{\pi^7}
+1\,024\frac{\beta(10)}{\pi^9}.
\]
\[
\int_0^\infty \frac{\tanh^9x}{x^3\cosh x}\,dx
=
-\frac{25\,609}{840}\frac{\beta(4)}{\pi^3}
+\frac{244\,312}{189}\frac{\beta(6)}{\pi^5}
-\frac{53\,536}{3}\frac{\beta(8)}{\pi^7}
+104\,448\frac{\beta(10)}{\pi^9}
-225\,280\frac{\beta(12)}{\pi^{11}}.
\]
\[
\int_0^\infty \frac{\tanh^9x}{x^5\cosh x}\,dx
=
\frac{25\,609}{126}\frac{\beta(6)}{\pi^5}
-\frac{488\,624}{27}\frac{\beta(8)}{\pi^7}
+428\,288\frac{\beta(10)}{\pi^9}
-3\,829\,760\frac{\beta(12)}{\pi^{11}}
+11\,714\,560\frac{\beta(14)}{\pi^{13}}.
\]
\[
\int_0^\infty \frac{\tanh^9x}{x^7\cosh x}\,dx
=
-\frac{51\,218}{45}\frac{\beta(8)}{\pi^7}
+\frac{7\,817\,984}{45}\frac{\beta(10)}{\pi^9}
-\frac{18\,844\,672}{3}\frac{\beta(12)}{\pi^{11}}
+79\,659\,008\frac{\beta(14)}{\pi^{13}}
-328\,007\,680\frac{\beta(16)}{\pi^{15}}.
\]
\[
\int_0^\infty \frac{\tanh^9x}{x^9\cosh x}\,dx
=
\frac{204\,872}{35}\frac{\beta(10)}{\pi^9}
-\frac{85\,997\,824}{63}\frac{\beta(12)}{\pi^{11}}
+69\,994\,496\frac{\beta(14)}{\pi^{13}}
-1\,194\,885\,120\frac{\beta(16)}{\pi^{15}}
+6\,372\,720\,640\frac{\beta(18)}{\pi^{17}}.
\]
\end{example}

\section{Explicit formulae for the Coefficients in Blagouchine's
Generalizations}
\label{sec:blagouchine-coefficients}

For \(a>0\), \(b>0\), and \(r\in\mathbb{N}\), consider the family
of logarithmic integrals
\[
\mathcal{I}_{r}(a,b)
:=
\int_{0}^{\infty}
\frac{\ln(x^{2}+a^{2})}{\cosh^{r}(bx)}\,dx.
\]

Blagouchine evaluated the first several members of this family and
observed that the general form of the answer depends on the parity of
\(r\) \cite{Blagouchine2014}. More precisely, he introduced rational
coefficients \(A_r\) and \(B_{r,\ell}\) in the corresponding odd and
even formulae. However, no closed formulae in terms of \(r\) and
\(\ell\) were given for the complete families \(A_r\) and
\(B_{r,\ell}\). Instead, the coefficients were described as rational
numbers whose values could be computed for each fixed \(r\), and a
table of values up to \(r=20\) was provided.

The coefficient-extraction notation
\(
[w^k]f(w)
\)
will denote the coefficient of \(w^k\) in the Taylor expansion of
\(f(w)\) at \(w=0\). We also write
\[
\Psi_q(z)
:=
\frac{d^q}{dz^q}\Psi(z),
\qquad
\Psi(z)=\frac{\Gamma'(z)}{\Gamma(z)},
\]
so that \(\Psi_0=\Psi\).

The coefficient-extraction formulae below make all the coefficients
in Blagouchine's two parity-dependent expressions explicit.

\subsection{Even powers}

Let \(n\in \mathbb N\). Then
\[
\boxed{
\begin{aligned}
\int_{0}^{\infty}
\frac{\ln(x^{2}+a^{2})}{\cosh^{2n}(bx)}\,dx
={}&
\frac{1}{b}
\frac{
2^{2n-1}\bigl((n-1)!\bigr)^2
}{
(2n-1)!
}
\ln\frac{\pi}{b}
\\
&+
\frac{2}{b}
\sum_{m=0}^{n-1}
\frac{
(-1)^m
[w^{2m}]
\left(\dfrac{w}{\sinh w}\right)^{2n}
}{
(2n-1-2m)!\,
\pi^{\,2n-2-2m}
}
\Psi_{2n-2-2m}\left(\frac12+\frac{ab}{\pi}\right).
\end{aligned}
}
\tag{E}
\]

Equivalently, the coefficients appearing in Blagouchine's notation
are
\[
\boxed{
A_{2n}
=
2(-1)^{n-1}
[w^{2n-2}]
\left(\frac{w}{\sinh w}\right)^{2n}=
\frac{
2^{2n-1}\bigl((n-1)!\bigr)^2
}{
(2n-1)!
}}
\]
and
\[
\boxed{
B_{2n,\ell}
=
\frac{
2(-1)^{n-1-\ell}
}{
(2\ell+1)!
}
[w^{2n-2-2\ell}]
\left(\frac{w}{\sinh w}\right)^{2n},
\qquad
1\leq \ell\leq n-1.
}
\]

\subsection{Odd powers}

Let \(n\in \mathbb N_0\), and define \(\displaystyle
\xi:=\frac14+\frac{ab}{2\pi},
\qquad
\eta:=\frac34+\frac{ab}{2\pi}.
\)
Then
\[
\boxed{
\begin{aligned}
\int_{0}^{\infty}
\frac{\ln(x^{2}+a^{2})}{\cosh^{2n+1}(bx)}\,dx
={}&
\frac{\pi(2n)!}{4^n b\,(n!)^2}
\ln\left[
\frac{2\pi}{b}
\left(
\frac{\Gamma(\eta)}{\Gamma(\xi)}
\right)^2
\right]
\\
&+
\frac{1}{b}
\sum_{m=0}^{n-1}
\frac{
(-1)^m
[w^{2m}]
\left(\dfrac{w}{\sinh w}\right)^{2n+1}
}{
(2n-2m)!\,
(2\pi)^{\,2n-2m-1}
}
\\
&\hspace{32mm}\times
\left\{
\Psi_{2n-2m-1}(\eta)
-
\Psi_{2n-2m-1}(\xi)
\right\}.
\end{aligned}
}
\tag{O}
\]

The finite sum in \((O)\) is understood to be empty when \(n=0\).

Equivalently, Blagouchine's coefficients in the odd case are
\[
\boxed{
A_{2n+1}
=
2(-1)^n
[w^{2n}]
\left(\frac{w}{\sinh w}\right)^{2n+1}=
2\,\frac{\binom{2n}{n}}{4^n}
}
\]
and
\[
\boxed{
B_{2n+1,\ell}
=
\frac{
(-1)^{n-\ell}
}{
2^{\,2\ell-1}(2\ell)!
}
[w^{2n-2\ell}]
\left(\frac{w}{\sinh w}\right)^{2n+1},
\qquad
1\leq \ell\leq n.
}
\]

Thus, formulae \((E)\) and \((O)\) replace the previously unspecified
rational coefficients by explicit coefficient extractions from the
single elementary generating function
\(\displaystyle
\left(\frac{w}{\sinh w}\right)^r.
\)

\section{Malmsten-type integrals for $\boldsymbol{\beta(2n)}$ and
$\boldsymbol{\zeta(2n+1)}$}
\label{sec:Malmsten-integrals}

Malmsten-type integrals involving the logarithmic kernel
$\ln\ln x$ have long been used to represent special values of
Dirichlet series. A classical example is the evaluation of
$\beta(2)$ obtained by Adamchik \cite{Adamchik1997},
\begin{equation}\label{eq:Adamchik-beta2}
\beta(2)
=
\frac{\pi}{2}
\int_{1}^{\infty}
\frac{x^{4}-6x^{2}+1}{(x^{2}+1)^{3}}
\ln(\ln x)\,dx.
\end{equation}
Similarly, Blagouchine \cite{Blagouchine2014} established the
representation
\begin{equation}\label{eq:Blagouchine-zeta3}
\zeta(3)
=
\frac{8\pi^{2}}{7}
\int_{1}^{\infty}
\frac{x(x^{4}-4x^{2}+1)}{(x^{2}+1)^{4}}
\ln(\ln x)\,dx.
\end{equation}

More recently, Kyrion \cite{Kyrion2025} generalized these isolated
evaluations to the complete families of even beta values and odd
zeta values. In particular, for $N\in\mathbb N$,
\begin{align}
\beta(2N)
&=
(-1)^{N}
\frac{\pi^{2N-1}}
     {2^{2N-1}(2N-1)!}
\int_{1}^{\infty}
\sum_{k=0}^{N-1}
\Bigg(
\sum_{j=0}^{k}
\binom{2k+1}{k-j}
(-1)^{j+1}(2j+1)^{2N-1}
\nonumber\\
&\hspace{5cm}\times
\frac{
\bigl((2k+1)x^{4}-2(2k+3)x^{2}+2k+1\bigr)x^{2k}
}
{(x^{2}+1)^{2k+3}}
\Bigg)
\ln(\ln x)\,dx,
\label{eq:Kyrion-beta}
\end{align}
and
\begin{align}
\zeta(2N+1)
&=
(-1)^{N}
\frac{2\pi^{2N}}
     {(2^{2N+1}-1)(2N)!}
\int_{1}^{\infty}
\sum_{k=1}^{N}
\Bigg(
\sum_{j=1}^{k}
\binom{2k}{k-j}
(-1)^{j}(2j)^{2N}
\nonumber\\
&\hspace{5cm}\times
\frac{
\bigl(2kx^{4}-2(2k+2)x^{2}+2k\bigr)x^{2k-1}
}
{(x^{2}+1)^{2k+2}}
\Bigg)
\ln(\ln x)\,dx.
\label{eq:Kyrion-zeta}
\end{align}

Thus, \cref{eq:Adamchik-beta2,eq:Blagouchine-zeta3} appear as the
first instances of general families. However, the price paid for the
generality of \cref{eq:Kyrion-beta,eq:Kyrion-zeta} is the occurrence
of nested finite sums, which considerably obscures the structure of
the corresponding kernels.

A much more compact representation was subsequently obtained in
\cite{TallaWaffo2025arxiv2511.02843}. There, the same two families
were written in terms of negative-order polylogarithms as
\begin{equation}\label{eq:compact-polylogarithmic-representations}
\begin{dcases}
\displaystyle
\int_{0}^{1}
\frac{\Im\!\left(\operatorname{Li}_{-2n}(ix)\right)}{x}
\ln\!\ln\frac{1}{x}\,dx
=
(-1)^{n+1}
\frac{2^{2n-1}(2n-1)!}{\pi^{2n-1}}
\beta(2n),
\\[2.2ex]
\displaystyle
\int_{0}^{1}
\frac{\operatorname{Li}_{-2n-1}(-x^{2})}{x}
\ln\!\ln\frac{1}{x}\,dx
=
(-1)^n
\left(1-2^{-(2n+1)}\right)
\frac{(2n)!}{2\pi^{2n}}
\zeta(2n+1).
\end{dcases}
\end{equation}

Compared with \cref{eq:Kyrion-beta,eq:Kyrion-zeta},
\cref{eq:compact-polylogarithmic-representations} eliminates the
nested combinatorial sums altogether and encodes the corresponding
Malmsten kernels in the two compact functions
$\Im\operatorname{Li}_{-2n}(ix)$ and
$\operatorname{Li}_{-2n-1}(-x^{2})$.

The proof given in \cite{TallaWaffo2025arxiv2511.02843}, however,
relied on interchanges of infinite summation and integration whose
validity was not established in that argument. The aim of the present
section is therefore to give a new and rigorous proof of
\cref{eq:compact-polylogarithmic-representations}, avoiding such
unjustified sum--integral interchanges. Our approach instead relies
on finite Eulerian expansions, shifted hyperbolic integrals,
Franssens' explicit differentiation formulae for \(\cot\) and
\(\csc\), and a coefficient-collapse phenomenon.

Throughout this section, we use the notation 
\(\displaystyle\left\langle {m\atop k}\right\rangle\) for Eulerian numbers of type \(A\) and 
\(\displaystyle\left\langle {m\atop k}\right\rangle^{B}\) for Eulerian numbers of type \(B\).

\begin{lemma}[Franssens half-sum identities]
\label[lemma]{lem:Franssens-half-sum-identities}
For every integer \(n\geq 1\), the following identities hold as
meromorphic identities in \(y\):
\begin{align*}
\sum_{k=0}^{n-1}
\left\langle {2n-1\atop k}\right\rangle^{B}
\cos\bigl((2n-2k-1)y\bigr)
&=
-2^{2n-2}\sin^{2n}y\,
\frac{d^{\,2n-1}}{dy^{\,2n-1}}\csc y,
\\
\sum_{k=0}^{n-1}
\left\langle {2n\atop k}\right\rangle
\cos\bigl((2n-2k-1)y\bigr)
&=
\frac{1}{2}\sin^{2n+1}y\,
\frac{d^{\,2n}}{dy^{\,2n}}\cot y.
\end{align*}
\end{lemma}

\begin{proof}
Franssens' differentiation formulae
\cite[eqs.~(112)--(113)]{Franssens2007}, specialized to the derivative
orders \(2n-1\) and \(2n\), give
\begin{align*}
\frac{d^{\,2n-1}}{dy^{\,2n-1}}\csc y
&=
-2^{-(2n-1)}
\frac{1}{\sin^{2n}y}
\sum_{k=0}^{2n-1}
\left\langle {2n-1\atop k}\right\rangle^{B}
\cos\bigl((2n-1-2k)y\bigr),
\\
\frac{d^{\,2n}}{dy^{\,2n}}\cot y
&=
\frac{1}{\sin^{2n+1}y}
\sum_{k=0}^{2n-1}
\left\langle {2n\atop k}\right\rangle
\cos\bigl((2n-1-2k)y\bigr).
\end{align*}
By the symmetry of the Eulerian and type-\(B\) Eulerian rows, together
with the evenness of the cosine function, each full sum is twice its
first half. Hence
\cref{lem:Franssens-half-sum-identities} follow.
\end{proof}

\begin{lemma}[Simple-pole derivative collapse]
\label[lemma]{lem:simple-pole-derivative-collapse}
Let \(a\in\mathbb C\), let \(m,q\in\mathbb N_0\), and let \(f\) be
meromorphic in a neighbourhood of \(a\) with at most a simple pole at
\(a\). Put \(\rho=\operatorname{Res}_{y=a}f(y)\), and let \(\Gamma\) be
a sufficiently small positively oriented contour around \(a\), containing
no other pole of \(f\). Then
\begin{equation*}
\frac{1}{2\pi i}\oint_\Gamma (y-a)^q f^{(m)}(y)\,dy
=
(-1)^m m!\,\rho\,\delta_{q,m}.
\label{eq:simple-pole-derivative-collapse}
\end{equation*}
\end{lemma}

\begin{proof}
Write \(f(y)=\dfrac{\rho}{y-a}+h(y)\), where \(h\) is holomorphic in a
neighbourhood of \(a\). Then
\(f^{(m)}(y)=(-1)^m m!\rho\,(y-a)^{-m-1}+h^{(m)}(y)\), and hence \(
(y-a)^q f^{(m)}(y)
=
(-1)^m m!\rho\,(y-a)^{q-m-1}
+
(y-a)^q h^{(m)}(y).
\) The second term is holomorphic at \(a\), whereas the first contributes a
\((y-a)^{-1}\)-term precisely when \(q=m\). Thus
\(\operatorname{Res}_{y=a}\bigl((y-a)^q f^{(m)}(y)\bigr)
=(-1)^m m!\rho\,\delta_{q,m}\), and the residue theorem yields the claim.
\end{proof}

\begin{lemma}[Eulerian--Chebyshev coefficient collapse]
\label[lemma]{lem:Eulerian-Chebyshev-collapse}
Let \(n,p\in \mathbb N\). Then
\begin{equation}
\begin{dcases}
\displaystyle
[z^{2n-1}]
\left\{
(\arcsin z)^{2p-1}
\sum_{k=0}^{n-1}
(-1)^k
\left\langle {2n-1\atop k}\right\rangle^{B}
U_{2n-2k-2}(z)
\right\}
=
(-1)^{n-1}2^{2n-2}(2n-1)!\,\delta_{p,n},
\\[2.2ex]
\displaystyle
[z^{2n}]
\left\{
(\arcsin z)^{2p}
\sum_{k=0}^{n-1}
(-1)^k
\left\langle {2n\atop k}\right\rangle
U_{2n-2k-2}(z)
\right\}
=
\frac{(-1)^{n-1}}{2}(2n)!\,\delta_{p,n}.
\end{dcases}
\label{eq:Eulerian-Chebyshev-collapse}
\end{equation}
Here \(\delta_{p,n}\) denotes the Kronecker delta.
\end{lemma}

\begin{proof}
For brevity, set
\[
\mathcal B_n(z)
:=
\sum_{k=0}^{n-1}
(-1)^k
\left\langle {2n-1\atop k}\right\rangle^{B}
U_{2n-2k-2}(z) \qquad \mathcal A_n(z)
:=
\sum_{k=0}^{n-1}
(-1)^k
\left\langle {2n\atop k}\right\rangle
U_{2n-2k-2}(z).
\]
Since \( U_{2n}(\sin y)
=
(-1)^{n}
\dfrac{\cos\bigl((2n+1)y\bigr)}{\cos y},
\)
we obtain
\begin{align*}
\mathcal B_n(\sin y)
&=
\frac{(-1)^{n-1}}{\cos y}
\sum_{k=0}^{n-1}
\left\langle {2n-1\atop k}\right\rangle^{B}
\cos\bigl((2n-2k-1)y\bigr),
\\
\mathcal A_n(\sin y)
&=
\frac{(-1)^{n-1}}{\cos y}
\sum_{k=0}^{n-1}
\left\langle {2n\atop k}\right\rangle
\cos\bigl((2n-2k-1)y\bigr).
\end{align*}

By \cref{lem:Franssens-half-sum-identities},
\[
\mathcal B_n(\sin y)
=
(-1)^n2^{2n-2}
\frac{\sin^{2n}y}{\cos y}
\frac{d^{\,2n-1}}{dy^{\,2n-1}}\csc y,
\qquad
\mathcal A_n(\sin y)
=
\frac{(-1)^{n-1}}{2}
\frac{\sin^{2n+1}y}{\cos y}
\frac{d^{\,2n}}{dy^{\,2n}}\cot y.
\label{eq:Eulerian-Chebyshev-trigonometric}
\]

Let \(C\) be a sufficiently small positively oriented circle around
\(0\), contained in the disk of analyticity of the principal branch of
\(\arcsin\), and set \(\Gamma=\arcsin(C).\)
Cauchy's coefficient formula, followed by the substitution
\(z=\sin y\), yields
\begin{align}
[z^{2n-1}]
\left\{
(\arcsin z)^{2p-1}\mathcal B_n(z)
\right\}
&=
\frac{(-1)^n2^{2n-2}}{2\pi i}
\oint_\Gamma
y^{2p-1}(\csc y)^{(2n-1)}\,dy,\label{eq:B-collapse-contour}
\\
[z^{2n}]
\left\{
(\arcsin z)^{2p}\mathcal A_n(z)
\right\}
&=
\frac{(-1)^{n-1}}{2}
\frac{1}{2\pi i}
\oint_\Gamma
y^{2p}(\cot y)^{(2n)}\,dy.
\label{eq:A-collapse-contour}
\end{align}

Since \(
\operatorname{Res}_{y=0}\csc y
=
\operatorname{Res}_{y=0}\cot y
=
1,
\)
\cref{lem:simple-pole-derivative-collapse}, applied with \(a=0\),
gives
\begin{align*}
\frac{1}{2\pi i}
\oint_\Gamma
y^{2p-1}(\csc y)^{(2n-1)}\,dy
&=
-(2n-1)!\,\delta_{2p-1,2n-1}=
-(2n-1)!\delta_{p,n}\,,
\\
\frac{1}{2\pi i}
\oint_\Gamma
y^{2p}(\cot y)^{(2n)}\,dy
&=
(2n)!\,\delta_{2p,2n}=
(2n)!\delta_{p,n}\,.
\end{align*}
Substitution into
\eqref{eq:B-collapse-contour}--\eqref{eq:A-collapse-contour}
gives \eqref{eq:Eulerian-Chebyshev-collapse}.
\end{proof}

\begin{proposition}[Polylogarithmic integral evaluations]
\label[proposition]{thm:polylogarithmic-integral-evaluations}
For every \(n\in\mathbb N\),
\begin{equation}
\begin{dcases}
\displaystyle
\int_{0}^{1}
\frac{\Im\!\left(\operatorname{Li}_{-2n}(ix)\right)}{x}
\ln\!\ln\frac{1}{x}\,dx
=
(-1)^{n-1}2^{2n-1}(2n-1)!
\frac{\beta(2n)}{\pi^{2n-1}},
\\[2.2ex]
\displaystyle
\int_{0}^{1}
\frac{\operatorname{Li}_{-2n-1}(-x^{2})}{x}
\ln\!\ln\frac{1}{x}\,dx
=
(-1)^n
\frac{(2^{2n+1}-1)(2n)!}{2^{2n+2}}
\frac{\zeta(2n+1)}{\pi^{2n}}.
\end{dcases}
\label{eq:polylogarithmic-integral-evaluations}
\end{equation}
\end{proposition}

\begin{proof}
We start from the identities established in
\cite[proof of Theorem 3.1]{TallaWaffo2026arxiv2602.16761}:
\begin{equation*}
\begin{dcases}
\displaystyle
\int_{0}^{1}
\frac{\Im\!\left(\operatorname{Li}_{-2n}(ix)\right)}{x}
\ln\!\ln\frac{1}{x}\,dx
=
-\int_{0}^{1}
\frac{1}{\ln x}\,
\frac{1}{(1+x^{2})^{2n}}
\sum_{k=0}^{n-1}
\left\langle {2n-1\atop k}\right\rangle^{B}
\left\{
(-x^{2})^{k}
+
(-x^{2})^{2n-1-k}
\right\}\,dx,
\\[2.2ex]
\displaystyle
\int_{0}^{1}
\frac{\operatorname{Li}_{-2n-1}(-x^{2})}{x}
\ln\!\ln\frac{1}{x}\,dx
=
\frac{1}{2}
\int_{0}^{1}
\frac{x}{\ln x}\,
\frac{1}{(1+x^{2})^{2n+1}}
\sum_{k=0}^{n-1}
\left\langle {2n\atop k}\right\rangle
\left\{
(-x^{2})^{k}
+
(-x^{2})^{2n-1-k}
\right\}\,dx.
\end{dcases}
\end{equation*}

Set \(
x=e^{-t}.
\) Then
\[
dx=-e^{-t}\,dt,
\qquad
\ln x=-t,
\qquad
1+e^{-2t}=2e^{-t}\cosh t.
\]
Moreover,
\[
(-e^{-2t})^{k}
+
(-e^{-2t})^{2n-1-k}
=
2(-1)^k e^{-(2n-1)t}
\sinh\bigl((2n-2k-1)t\bigr).
\]
Consequently,
\begin{equation*}
\begin{dcases}
\displaystyle
\int_{0}^{1}
\frac{\Im\!\left(\operatorname{Li}_{-2n}(ix)\right)}{x}
\ln\!\ln\frac{1}{x}\,dx
=
\frac{1}{2^{2n-1}}
\sum_{k=0}^{n-1}
(-1)^k
\left\langle {2n-1\atop k}\right\rangle^{B}
\int_{0}^{\infty}
\frac{\sinh\bigl((2n-2k-1)t\bigr)}
{t\cosh^{2n}t}\,dt,
\\[2.2ex]
\displaystyle
\int_{0}^{1}
\frac{\operatorname{Li}_{-2n-1}(-x^{2})}{x}
\ln\!\ln\frac{1}{x}\,dx
=
-\frac{1}{2^{2n+1}}
\sum_{k=0}^{n-1}
(-1)^k
\left\langle {2n\atop k}\right\rangle
\int_{0}^{\infty}
\frac{\sinh\bigl((2n-2k-1)t\bigr)}
{t\cosh^{2n+1}t}\,dt.
\end{dcases}
\end{equation*}

Applying
\cref{prop:shifted_integrals_chebyshev_arcsine}
with \(j=n-k-1\) and interchanging the finite sums, we obtain
\begin{equation}
\begin{dcases}
\displaystyle
\int_{0}^{1}
\frac{\Im\!\left(\operatorname{Li}_{-2n}(ix)\right)}{x}
\ln\!\ln\frac{1}{x}\,dx
=
\sum_{p=1}^{n}
2^{\,2p-2n+1}
\frac{\beta(2p)}{\pi^{2p-1}}
[z^{2n-1}]
\left\{
(\arcsin z)^{2p-1}
\sum_{k=0}^{n-1}
(-1)^k
\left\langle {2n-1\atop k}\right\rangle^{B}
U_{2n-2k-2}(z)
\right\},
\\[2.2ex]
\displaystyle
\int_{0}^{1}
\frac{\operatorname{Li}_{-2n-1}(-x^{2})}{x}
\ln\!\ln\frac{1}{x}\,dx
=
-\sum_{p=1}^{n}
\frac{2^{2p+1}-1}{2^{2n+1}}
\frac{\zeta(2p+1)}{\pi^{2p}}
[z^{2n}]
\left\{
(\arcsin z)^{2p}
\sum_{k=0}^{n-1}
(-1)^k
\left\langle {2n\atop k}\right\rangle
U_{2n-2k-2}(z)
\right\}.
\end{dcases}
\label{eq:polylogarithmic-coefficient-reduction}
\end{equation}

By \cref{lem:Eulerian-Chebyshev-collapse}, the coefficient expressions
in \eqref{eq:polylogarithmic-coefficient-reduction} vanish for every
\(p\neq n\). Thus only the term \(p=n\) survives in each sum, and
\eqref{eq:polylogarithmic-integral-evaluations} follows.
\end{proof}

\section*{Acknowledgments}

The author acknowledges the use of an AI language model for assistance with the presentation of the manuscript, numerical experimentation, and verification of elementary asymptotic estimates.

\printbibliography

\end{document}